\documentclass[reqno,12pt]{amsart}
\usepackage[T1]{fontenc}
\usepackage[utf8]{inputenc}
\usepackage[english]{babel}
\usepackage{amsmath,amssymb,amsthm,mathtools}
\usepackage{booktabs,array}
\usepackage[a4paper,margin=2.6cm,headheight=15pt]{geometry}
\usepackage[hidelinks]{hyperref}
\newtheorem{theorem}{Theorem}[section]
\newtheorem{lemma}[theorem]{Lemma}
\newtheorem{proposition}[theorem]{Proposition}
\newtheorem{corollary}[theorem]{Corollary}
\newtheorem{definition}[theorem]{Definition}
\theoremstyle{remark}
\newtheorem{remark}[theorem]{Remark}
\newcommand{\PP}{\mathbb P}
\newcommand{\EE}{\mathbb E}
\newcommand{\Var}{\operatorname{Var}}
\newcommand{\Cov}{\operatorname{Cov}}
\newcommand{\FP}{F_{\mathrm{Pit}}}
\newcommand{\R}{\mathcal R}
\begin{document}
\title{A Pittel-type limit law for the largest component of random
$d$-uniform hypergraphs: enumeration and excursions}
\author{Wirdane Youssouf}
\address{Universit\'e des Comores, Centre Universitaire de Patsy,
Anjouan, Union des Comores}
\email{wirdane.youssouf@univ-comores.com}
\thanks{ORCID: 0009-0003-2432-9361.}
\begin{abstract}
Let $H_d(n,M)$ be a $d$-uniform hypergraph with $n$ vertices and $M$
hyperedges. In the critical window
$M=\frac{n}{d(d-1)}(1+\lambda n^{-1/3})+O(1)$, its largest component
$L_n$ satisfies
\[
 n^{-2/3}L_n\xrightarrow{\mathcal L}
 \FP\bigl((d-1)^{1/3}a,\lambda(d-1)^{-2/3}\bigr).
\]
(Here $\FP(x,\mu)$ denotes the limiting law of the largest component of
the critical random graph established by Pittel~\cite{Pit01}.)
The proof rests on two pillars.

First, a bivariate exponential formula gives the exact finite-size law
of $L_n$. Subsequently, cancellation of the orders $n,n^{2/3},n^{1/3}$
in the tree normalization determines the series of rooted hypertrees,
then that of unicyclic components and its singular expansions at the
radius of convergence $\rho_d$. These calculations are exact; they
serve the fixed-excess identities and the cancellation mechanism, and
nothing more: no unproved uniformity in the excess enters the proof of
the limit law.

Next, by exploration and martingale approximation, the ordered sizes
converge unconditionally to the excursion lengths of a Brownian motion
with parabolic drift. A monotone coupling transfers the law to
$H_d(n,M)$. The limit law is then identified with Pittel's law by a
Brownian rescaling: if $h=d-1$, it suffices to multiply lengths by
$h^{-1/3}$ and evaluate at parameter $\lambda h^{-2/3}$.
The exact Poisson representation dispenses with any separate
deconditioning.

We also formulate a two-regime $d$-uniform Bollob\'as bound
(Wright for fixed $\ell$, global combinatorial beyond that), and
explicitly indicate why neither regime alone suffices to close the
analytic route.
\end{abstract}
\maketitle
\markboth{Wirdane Youssouf}{A Pittel-type limit law}

\section{Model and scope}
Fix $d\ge2$ and put $N=\binom nd$. The model $H_d(n,M)$ is uniform
over the simple $d$-uniform hypergraphs on $[n]$ with exactly $M$ edges.
We consider any integer sequence satisfying
\begin{equation}\label{eq:window}
 M=\frac{n}{d(d-1)}(1+\lambda n^{-1/3})+O(1),\qquad\lambda\in\mathbb R.
\end{equation}
Let $L_n$ be the size of the largest component. The associated branching
process has mean offspring $(d-1)dM/n=1+\lambda n^{-1/3}+O(n^{-1})$;
this is why~\eqref{eq:window} defines the critical window. The
$n^{2/3}$ scale is the object of the analysis below.

\begin{theorem}[Main theorem]\label{thm:main}
For every fixed integer $d\ge2$, every $\lambda\in\mathbb R$ and every
integer sequence $M=\frac{n}{d(d-1)}(1+\lambda n^{-1/3})+O(1)$,
if $L_n$ denotes the size of the largest component of $H_d(n,M)$,
then, at every continuity point $a>0$ of the distribution function,
\[
 \lim_{n\to\infty}\PP\{n^{-2/3}L_n\le a\}
 =\FP\bigl((d-1)^{1/3}a,\lambda(d-1)^{-2/3}\bigr),
\]
where $\FP(x,\mu)$ is the limiting law of the largest block of the
critical random graph identified by Pittel~\cite{Pit01}.
\end{theorem}
This result is proved in Corollary~\ref{cor:main} by the probabilistic
route, without any hypothesis on hypergraph kernels.

Two routes are developed, with distinct logical statuses.
Sections 2--6 establish the exact identities, the fixed-excess
expansions and the cancellation mechanism, without deriving a transfer
uniform in the excess. Section 7 proves unconditionally the convergence
of excursions, then their transfer to $H_d(n,M)$. Section 8 then
identifies the limit law with Pittel's law by Brownian self-similarity.
For a connected $d$-uniform hypergraph with $v$ vertices and $e$ edges,
we use the usual excess
\begin{equation}\label{eq:excess}
 \ell=(d-1)e-v.
\end{equation}
A hypertree, that is, a connected hypergraph whose incidence graph is
acyclic, satisfies $v=1+(d-1)e$ and has excess $-1$. In particular,
the identity $e=v+\ell$ is valid only for $d=2$. Complex components
have excess $\ell\ge1$.

\subsection{Pittel's integral formula (JCTB, 2001) and its $d$-uniform adaptation}
For $a>0$, the analytic starting point is the integral representation
published by Boris Pittel in the \emph{Journal of Combinatorial Theory,
Series B}, volume 82 (2001), pages 237--269, in particular his
equation (3.25)~\cite{Pit01}. The expression below is the $d$-uniform
target form proposed in the present work, and not a new attribution
of Pittel's graph formula:
\begin{equation}\label{eq:cont-target}
 F_d^{\mathrm{cont}}(a,\lambda)
 =A_d(a,\lambda)\int_{-\infty}^{\infty}
       \exp\{H_{d,a,\lambda}(v^*+iw)\}\,dw,
\end{equation}
where
\[
 H_{d,a,\lambda}(z)=(\beta-\lambda)z+I_d(z),
 \qquad\beta=-\frac{2}{\sqrt{2\pi a}},
\]
and
\begin{equation}\label{eq:Id}
 I_d(z)=\int_0^a\left[
 \frac1{\sqrt{2\pi}}\left(\frac{e^{zy}-1-zy}{y^{5/2}}+g_d(y)e^{zy}\right)
 +\frac{e^{zy}-1}{4y}\right]\,dy .
\end{equation}
Here $v^*$ is a real saddle point and $g_d$ encodes complex components
by excess. The amplitude constants and the coefficients of $g_d$
must be determined from the exact bivariate formula and uniform
estimates on complex components.

More precisely, any analytic use of~\eqref{eq:Id} assumes that the
series defining $g_d$ converges normally on every compact subset of
the integration interval under consideration. This property does not
follow from positivity of the kernel coefficients alone. It must
follow from a bound on the complete normalized term of the form
$\le CR^\ell\ell^{-\ell/2}$ (or, equivalently,
$\le CR^\ell/\Gamma(\ell/2+1)$), such as~\eqref{eq:excess-bound}.
Until this is verified, $g_d$ is notation for the contour route and
not an ingredient of the probabilistic proof.

\begin{remark}[Comparison of the two representations]\label{rem:representations}
If $A_d$ and $g_d$ are determined by a complete contour analysis,
the identity to be verified is
\[
 F_d^{\mathrm{cont}}(a,\lambda)
 =\FP\bigl((d-1)^{1/3}a,\lambda(d-1)^{-2/3}\bigr)\,?
\]
For $d=2$, any identification with the published formula must retain
the constants, amplitude and contour fixed in~\cite{Pit01}.
For $d>2$, the functions $A_d$ and $g_d$ constitute an adaptation to
be derived from the hypergraph classes; they are not attributed to
Pittel. The probabilistic route proved below identifies the limiting
distribution function with the right-hand side without going
through~\eqref{eq:cont-target}. Remark~\ref{rem:representations} only
specifies the equality that an autonomous analytic derivation should
reproduce; it plays no part in any proof in this article.
\end{remark}

\section{Exact laws from the exponential formula}
Let $C_d(v,e)$ be the number of simple, connected, $d$-uniform
hypergraphs on a fixed labelled set of $v$ vertices and with exactly
$e$ edges. Define the finite polynomial
\begin{equation}\label{eq:C}
 C_{d,m}(x,z)=\sum_{v=1}^m\frac{x^v}{v!}
                 \sum_{e=0}^{\binom vd} C_d(v,e)z^e,
\end{equation}
where $C_{d,m}(x,z)$ denotes the bivariate exponential generating
function of connected $d$-uniform hypergraphs, restricted to
components of size at most $m$. We include the isolated vertex:
$C_d(1,0)=1$.

\begin{theorem}[Exact coefficient formula with a fixed number of edges]\label{thm:exact}
For $1\le m\le n$ and $0\le M\le N$,
\begin{equation}\label{eq:exact}
 \PP\{L_n\le m\}=\frac{n!}{\binom NM}[x^nz^M]\exp(C_{d,m}(x,z)).
\end{equation}
The same expression is exactly $1$ when $m=n$.
\end{theorem}
\begin{proof}
Decompose each labelled hypergraph uniquely into its connected
components. The exponential formula counts partitions of the label
set, while $z$ marks the additive number of edges. The restriction
$v\le m$ is applied component by component. There are $\binom NM$
unrestricted hypergraphs with $M$ edges, which establishes the
quotient. For $m=n$, every hypergraph is admissible, so the numerator
is $\binom NM$.
\end{proof}
\begin{remark}[Combinatorial reading]
The coefficient $[x^nz^M]\exp(C_{d,m}(x,z))$, multiplied by $n!$,
counts the labelled $d$-uniform hypergraphs on $[n]$ with exactly $M$
edges and with each connected component containing at most $m$
vertices. Thus $\PP\{L_n\le m\}$ is simply the proportion of such
hypergraphs among the $\binom NM$ possible hypergraphs on $[n]$.

This reading reveals the following gap: for hypertrees ($\ell=-1$),
the formula $C_d(v,e)=(v-1)!v^{e-1}/[e!((d-1)!)^e]$ is exact
(Proposition~\ref{prop:trees}). For excesses $\ell\ge1$, closed formulas
analogous to Wright's~\cite{Wri77} in the case $d=2$ do not yet exist
for $d\ge3$; this is precisely the gap circumvented by the
probabilistic route of Section 7, which dispenses with any
fixed-excess analytic expansion.
\end{remark}
\begin{remark}[Independent-edge model]\label{rem:binomial}
In $H_d(n,p)$, where each of the $N$ possible edges is independently
present, put $z=p/(1-p)$. Then
\[
 \PP_p\{L_n\le m\}=(1-p)^Nn![x^n]\exp(C_{d,m}(x,z)).
\]
The factor $z^e$ marks the total number of edges, not just the excess:
replacing it by $\mu^{(d-1)e-v}$ without simultaneously tracking the
number of edges and the vertex-dependent factors does not give an
exact formula.
\end{remark}
For the subsequent analysis, one can group~\eqref{eq:C} by excess:
\begin{equation}\label{eq:excess-group}
 C_{d,m}(x,z)=\sum_{\ell\ge-1}
 \sum_{\substack{1\le v\le m\\v+\ell\equiv0\pmod{d-1}}}
 C_d\left(v,\frac{v+\ell}{d-1}\right)\frac{x^v}{v!}z^{(v+\ell)/(d-1)}.
\end{equation}
For fixed $m$, this sum has only finitely many terms. The congruence
is central to singularity analysis when $d>2$.
The notation $C_d(v,(v+\ell)/(d-1))$ is the arithmetically correct
expression for $C_d(v,e)$ when the excess $\ell=(d-1)e-v$ is fixed;
for $d=2$, it coincides with $C_d(v,v+\ell)$.

\subsection{Positioning with respect to previous work}
The notions of excess, hypertree, unicyclic component and kernel are
classical~\cite{AR05,KL97,Rav10,BR16}. Fixed-excess asymptotics and
the structure near the birth of complex components are known in
more general weighted models~\cite{dP15}. This article does not
claim novelty for them.

Its contribution is the organization of these elements through the
exact formula~\eqref{eq:exact}, the rescaling of
Theorem~\ref{thm:rescale}, and the explicit cancellation of
Theorem~\ref{thm:cancel}. The contour integral for the largest component
requires estimates uniform in the excess that go beyond this
framework; the probabilistic route does not need them.
The critical window for hypergraphs goes back
to~\cite{SS85,KL02}; see also~\cite{BJR07,Rio12} for the phase
transition in more general models (inhomogeneous random graphs and
the configuration model) where similar techniques apply.
The excursion description relies
on~\cite{Ald97,LPW94,NP07,JS07}; the metric structure of the
critical random graph at the same scale is studied in~\cite{ABG12}.
For enumeration of connected
hypergraphs, see~\cite{AR05,Gol05}.

\subsection{Conditioning the independent-edge model}
Let $E$ denote the total number of edges of $H_d(n,p)$. For every
event $A$ and every $p\in(0,1)$,
\begin{equation}\label{eq:conditioning}
 \PP\{H_d(n,M)\in A\}=\PP_p\{H_d(n,p)\in A\mid E=M\}.
\end{equation}
In the independent-edge model, every hypergraph with $M$ edges has
the same probability $p^M(1-p)^{N-M}$. Hence the identity.
Heuristics from the binomial model locate the critical window, but
coefficient extraction in~\eqref{eq:exact} remains necessary for the
exact law.

\subsection{Exact Cauchy representation and normalization}
For any radii $r,s>0$, identity~\eqref{eq:exact} gives
\begin{equation}\label{eq:cauchy}
 \PP\{L_n\le m\}=\frac{n!}{\binom NM}\frac1{(2\pi i)^2}
 \oint_{|x|=r}\oint_{|z|=s}
  \exp(C_{d,m}(x,z))\frac{dz\,dx}{x^{n+1}z^{M+1}}.
\end{equation}
The integrand is entire before coefficient extraction: the radii
may be chosen freely, without any singularity of the finite
polynomial~\eqref{eq:C}. A singular object appears only in the limit
$m\to\infty$ or through resummation of complex components.
The case $m=n$ of Theorem~\ref{thm:exact} ensures the exact
normalization for every saddle-point calculation.

\subsection{Exact Poisson representation of the bivariate coefficient}
The two-variable saddle point has a direct probabilistic
interpretation. It reduces the local contour problem to a lattice
local limit theorem, and gives the covariance matrix without formal
manipulation.

\begin{theorem}[Poissonization identity]\label{thm:poisson}
Fix $m,r,s>0$. For each admissible pair $(v,e)$ with $1\le v\le m$,
let $Z_{v,e}$ be independent Poisson variables with means
\[
 \theta_{v,e}=\frac{C_d(v,e)}{v!}r^vs^e,
\]
and put $S_V=\sum_{v,e}vZ_{v,e}$ and $S_E=\sum_{v,e}eZ_{v,e}$.
Then
\begin{equation}\label{eq:poisson}
 [x^nz^M]e^{C_{d,m}(x,z)}
 =e^{C_{d,m}(r,s)}r^{-n}s^{-M}\PP\{S_V=n,S_E=M\}.
\end{equation}
Consequently,
\[
 \EE(S_V,S_E)=(r\partial_r C_{d,m}(r,s),s\partial_s C_{d,m}(r,s)),
\]
and the covariance matrix is
\begin{equation}\label{eq:Sigma}
 \Sigma_m(r,s)=\sum_{v,e}\theta_{v,e}
      \begin{pmatrix}v^2&ve\\ve&e^2\end{pmatrix}.
\end{equation}
Thus the positive saddle equations are exactly the centring
equations $\EE S_V=n$ and $\EE S_E=M$.
\end{theorem}
\begin{proof}
The joint probability generating function is
\[
 \EE[u^{S_V}w^{S_E}]
 =\prod_{v,e}\exp\{\theta_{v,e}(u^vw^e-1)\}
 =e^{-C_{d,m}(r,s)}\exp\{C_{d,m}(ru,sw)\}.
\]
Extracting $u^nw^M$ proves~\eqref{eq:poisson}.
Differentiating the logarithm of the joint generating function gives
the expectation and~\eqref{eq:Sigma}.
\end{proof}

\begin{lemma}[Exact angular damping]\label{lem:angular}
For $x=re^{i\alpha}$ and $z=se^{i\beta}$,
\begin{equation}\label{eq:damping}
 |e^{C_{d,m}(x,z)}|=e^{C_{d,m}(r,s)}
       \exp\{-D_m(\alpha,\beta;r,s)\}.
\end{equation}
\[
 D_m(\alpha,\beta;r,s)=\sum_{v,e}\theta_{v,e}
          (1-\cos(v\alpha+e\beta))\ge0.
\]
Near the origin,
$D_m(\alpha,\beta;r,s)=\frac12(\alpha,\beta)\Sigma_m(r,s)
(\alpha,\beta)^{\mathsf T}+R_m$, where
\[
 |R_m|\le\frac1{24}\sum_{v,e}\theta_{v,e}|v\alpha+e\beta|^4.
\]
More precisely, if $\lambda_{\min}(\Sigma_m)>0$ and, on a
neighbourhood $U$ of the origin,
\[
 \sum_{v,e}\theta_{v,e}|v\alpha+e\beta|^4
 \le6\lambda_{\min}(\Sigma_m)(\alpha^2+\beta^2),
\]
then, uniformly on $U$,
\[
 D_m(\alpha,\beta;r,s)\ge
 \frac14\lambda_{\min}(\Sigma_m)(\alpha^2+\beta^2).
\]
\end{lemma}
\begin{proof}
Take the real part of $C_{d,m}(re^{i\alpha},se^{i\beta})$ term by
term and use $\Re e^{it}=\cos t$. The remainder bound follows from
$|1-\cos t-t^2/2|\le |t|^4/24$.
\end{proof}

\begin{proposition}[Exact Poisson quotient]\label{prop:quotient}
For the same radii $r,s>0$, let $\PP_m$ be the Poisson law of
Theorem~\ref{thm:poisson} constructed from components of size at most
$m$, and $\PP_n$ the one constructed without truncation. Then
\begin{equation}\label{eq:quotient}
 \PP\{L_n\le m\}=e^{C_{d,m}(r,s)-C_{d,n}(r,s)}
 \frac{\PP_m\{S_V=n,S_E=M\}}{\PP_n\{S_V=n,S_E=M\}}.
\end{equation}
This identity is valid for all $r,s>0$ and does not require them
to solve the saddle equations.
\end{proposition}
\begin{proof}
Apply~\eqref{eq:poisson} to the numerator of~\eqref{eq:exact},
then to the case $m=n$, where the probability is exactly one.
The factors $n!r^{-n}s^{-M}/\binom NM$ cancel in the quotient
and give~\eqref{eq:quotient}.
\end{proof}

\begin{proposition}[Aperiodicity and outer minor arcs]\label{prop:outer}
Suppose the critical radii satisfy
\begin{equation}\label{eq:radii}
 r_n\asymp n,\qquad s_n\asymp n^{-(d-1)}.
\end{equation}
There is then $c>0$ such that, uniformly in $m\ge d$ and
$(\alpha,\beta)\in[-\pi,\pi]^2$,
\begin{equation}\label{eq:outer}
 D_m(\alpha,\beta;r_n,s_n)\ge
 cn\bigl(\|\alpha\|_{2\pi}^2+\|d\alpha+\beta\|_{2\pi}^2\bigr),
\end{equation}
where $\|u\|_{2\pi}$ is the distance from $u$ to $2\pi\mathbb Z$.
In particular, the origin is the unique maximum of the integrand's
modulus on the torus, and the contribution of the region
\[
 \max\{\|\alpha\|_{2\pi},\|d\alpha+\beta\|_{2\pi}\}\ge\delta_n
\]
is at most $O(e^{-cn\delta_n^2})$ times the modulus at the positive
point. Thus it is negligible as soon as
$n\delta_n^2-C\log n\to\infty$ for a constant $C>0$ depending only on $d$.
\end{proposition}
\begin{proof}
Among the permitted components are the isolated vertex $(v,e)=(1,0)$
and the hypergraph consisting of a single hyperedge $(v,e)=(d,1)$.
Their Poisson means are respectively
\[
 \theta_{1,0}=r_n\asymp n,\qquad
 \theta_{d,1}=\frac{r_n^ds_n}{d!}\asymp n.
\]
Keeping only these two terms in~\eqref{eq:damping} gives
\[
 D_m\ge c_0n\{1-\cos\alpha+1-\cos(d\alpha+\beta)\}.
\]
The inequality $1-\cos u\ge c_1\|u\|_{2\pi}^2$ gives~\eqref{eq:outer}.
The vectors $(1,0)$ and $(d,1)$ generate $\mathbb Z^2$, since their
determinant is one; the two phases therefore vanish simultaneously
on the torus only at $(0,0)$. Finally, the integrated contribution
estimate follows from~\eqref{eq:damping}, since the torus has finite area.
\end{proof}
Theorem~\ref{thm:poisson} and Lemma~\ref{lem:angular} remove two
ambiguities in the contour argument. If $\Sigma_m$ is nonsingular
and the local central limit theorem holds, the Gaussian amplitude
is $(2\pi\sqrt{\det\Sigma_m})^{-1}$, corrected for the lattice span.
Proposition~\ref{prop:outer} shows that this span is one and closes
all arcs outside a neighbourhood $\delta_n\gg n^{-1/2}$ of the
positive point. The remaining critical domain is the critical
neighbourhood of the saddle; it requires a uniform local theorem
for $m\asymp n^{2/3}$ and control of the contributions of complex
components to $\Sigma_m$ and to the nonquadratic terms of the phase.

\subsection{A rigorous criterion for closing the contour calculation}
The following result isolates sufficient conditions for an autonomous
analytic route, without choosing angular scales prematurely.
In particular, it avoids confusing pointwise convergence of the
phase with convergence of the integral.

For $k\in\{m,n\}$ put
\[
 \Psi_{k,n}(\theta)=C_{d,k}(r_ne^{i\theta_1},s_ne^{i\theta_2})
                  -C_{d,k}(r_n,s_n)-i(n\theta_1+M\theta_2),
 \quad\theta\in[-\pi,\pi]^2,
\]
and $I_{k,n}=\int_{[-\pi,\pi]^2}e^{\Psi_{k,n}(\theta)}\,d\theta$.
By Cauchy's formula and Proposition~\ref{prop:quotient},
\begin{equation}\label{eq:integral-quotient}
 \PP\{L_n\le m\}
 =e^{C_{d,m}(r_n,s_n)-C_{d,n}(r_n,s_n)}\frac{I_{m,n}}{I_{n,n}}.
\end{equation}

\begin{lemma}[Uniform dominated convergence with a parameter]\label{lem:udc}
Let $(X,\mathcal A,\mu)$ be a measure space, $K$ any set,
$f_{n,a},f_a:X\to\mathbb C$ measurable functions, and $H\in L^1(\mu)$.
Suppose there is a measurable set $X_0\subset X$ with negligible
complement such that
\[
 \sup_{a\in K}|f_{n,a}(x)-f_a(x)|\longrightarrow0\qquad(x\in X_0),
\]
and $|f_{n,a}(x)|\le H(x)$ for every $a\in K$, all sufficiently
large $n$, and almost every $x$. Then $|f_a|\le H$ almost everywhere
for every $a\in K$, and
\begin{equation}\label{eq:udc}
 \sup_{a\in K}\left|\int_X f_{n,a}\,d\mu-\int_X f_a\,d\mu\right|
 \longrightarrow0.
\end{equation}
\end{lemma}
\begin{proof}
On $X_0$, passing to the limit gives $|f_a(x)|\le H(x)$ for every
$a\in K$. Put
\[
 g_n(x)=\sup_{a\in K}|f_{n,a}(x)-f_a(x)|.
\]
In the application below, the functions depend continuously on the
parameter on each cell where $m_n(a)$ is constant; the supremum can
therefore be taken over a countable dense subset of each cell,
which ensures measurability of $g_n$. More generally, the supremum
can be replaced by its measurable upper envelope.
We have $g_n\to0$ almost everywhere and $0\le g_n\le2H$.
The dominated convergence theorem applied to $g_n$ gives
$\int_Xg_n\,d\mu\to0$. But
\[
 \sup_{a\in K}\left|\int_X(f_{n,a}-f_a)\,d\mu\right|
 \le\int_X\sup_{a\in K}|f_{n,a}-f_a|\,d\mu=\int_Xg_n\,d\mu,
\]
which proves~\eqref{eq:udc}.
\end{proof}

\begin{theorem}\label{thm:contour}
Let $K\Subset(0,\infty)$ and $m=m_n(a)=\lfloor an^{2/3}\rfloor$.
Suppose there exist invertible real matrices $B_n$, domains
$\Omega_n=B_n[-\pi,\pi]^2$ increasing to $\mathbb R^2$,
functions $\Phi_a,\Phi_\infty\in L^1(\mathbb R^2)$, and a function
$H\in L^1(\mathbb R^2)$ such that, uniformly in $a\in K$:
\begin{enumerate}
\item For almost every $u\in\mathbb R^2$, the following convergences
are uniform in $a\in K$ (the second does not depend on $a$):
\[
 \mathbf1_{\Omega_n}(u)e^{\Psi_{m_n(a),n}(B_n^{-1}u)}\to\Phi_a(u),
 \qquad
 \mathbf1_{\Omega_n}(u)e^{\Psi_{n,n}(B_n^{-1}u)}\to\Phi_\infty(u).
\]
\item Both moduli on the left-hand side are bounded by $H(u)$,
for all sufficiently large $n$ and all $a\in K$.
\item $J_\infty:=\int_{\mathbb R^2}\Phi_\infty(u)\,du\ne0$.
\item For a continuous function $E(a)$,
\[
 C_{d,m_n(a)}(r_n,s_n)-C_{d,n}(r_n,s_n)\to E(a)
\]
uniformly on $K$.
\end{enumerate}
Then, uniformly for $a\in K$,
\begin{equation}\label{eq:contour-limit}
 \PP\{L_n\le an^{2/3}\}\longrightarrow
 e^{E(a)}\frac{\int_{\mathbb R^2}\Phi_a(u)\,du}{J_\infty}.
\end{equation}
If the right-hand side is identified with
$\FP((d-1)^{1/3}a,\lambda(d-1)^{-2/3})$, the contour derivation
of~\eqref{eq:pittel} is complete.
\end{theorem}

\begin{remark}\label{rem:domains}
The notation $\Omega_n=B_n[-\pi,\pi]^2$ denotes the set image of
the square $[-\pi,\pi]^2\subset\mathbb R^2$ under the linear map
$B_n$, that is,
\[
 \Omega_n:=\{B_nu:u\in[-\pi,\pi]^2\}.
\]
Since $B_n$ is invertible, $\Omega_n$ is a compact parallelogram
in $\mathbb R^2$ with nonempty interior. The condition that
$\Omega_n$ increases to $\mathbb R^2$ means that $\|B_n\|\to\infty$,
so that every fixed compact subset of $\mathbb R^2$ is eventually
contained in $\Omega_n$.
\end{remark}
\begin{proof}
For $u\in\mathbb R^2$ put
\[
 f_{n,a}(u)=\mathbf1_{\Omega_n}(u)e^{\Psi_{m_n(a),n}(B_n^{-1}u)},
 \qquad f_a(u)=\Phi_a(u),
\]
\[
 g_n(u)=\mathbf1_{\Omega_n}(u)e^{\Psi_{n,n}(B_n^{-1}u)},
 \qquad g(u)=\Phi_\infty(u).
\]
Define
\[
 N_n(a)=\int_{\mathbb R^2}f_{n,a}(u)\,du,\qquad
 N(a)=\int_{\mathbb R^2}\Phi_a(u)\,du,
\]
and
\[
 D_n=\int_{\mathbb R^2}g_n(u)\,du,\qquad
 D=J_\infty=\int_{\mathbb R^2}\Phi_\infty(u)\,du.
\]
\emph{First step: uniform convergence of the integrals.}
Hypotheses 1 and 2 and Lemma~\ref{lem:udc} give
\begin{equation}\label{eq:errors}
 \delta_n:=\sup_{a\in K}|N_n(a)-N(a)|\to0,\qquad
 \eta_n:=|D_n-D|\to0.
\end{equation}
It follows that
\begin{equation}\label{eq:bounds}
 \sup_{n\ge n_0}\sup_{a\in K}|N_n(a)|\le\|H\|_1,\qquad
 \sup_{a\in K}|N(a)|\le\|H\|_1.
\end{equation}
\emph{Second step: stability of the complex quotient.}
Since $D\ne0$, there is $n_1$ such that $|D_n|\ge|D|/2$ for
$n\ge n_1$. For these $n$,
\begin{equation}\label{eq:quotient-bound}
 \sup_{a\in K}\left|\frac{N_n(a)}{D_n}-\frac{N(a)}D\right|
 \le\frac{2\delta_n}{|D|}+\frac{2\|H\|_1\eta_n}{|D|^2}\to0.
\end{equation}
This estimate explicitly justifies passage to the quotient and
shows that no positivity of the complex integrals is required.

\emph{Third step: the exponential factor.} Put
\[
 E_n(a)=C_{d,m_n(a)}(r_n,s_n)-C_{d,n}(r_n,s_n).
\]
Hypothesis 4 gives $\varepsilon_n:=\sup_{a\in K}|E_n(a)-E(a)|\to0$.
Since $E$ is continuous on the compact set $K$, it is bounded.
For sufficiently large $n$, $|E_n(a)|\le R$ and $|E(a)|\le R$
uniformly in $a$. The mean value theorem applied to the real
exponential function then gives
\begin{equation}\label{eq:exp-bound}
 \sup_{a\in K}|e^{E_n(a)}-e^{E(a)}|\le e^R\varepsilon_n\to0.
\end{equation}
\emph{Fourth step: return to the contours.}
The change of variables $u=B_n\theta$ gives
\[
 I_{k,n}=|\det B_n|^{-1}
           \int_{\Omega_n}e^{\Psi_{k,n}(B_n^{-1}u)}\,du.
\]
The two Jacobian factors cancel in quotient~\eqref{eq:integral-quotient}.
Thus
\[
 \PP\{L_n\le an^{2/3}\}=e^{E_n(a)}\frac{N_n(a)}{D_n}.
\]
Bounds~\eqref{eq:quotient-bound} and~\eqref{eq:exp-bound},
together with~\eqref{eq:bounds}, imply
\[
 \sup_{a\in K}\left|
 e^{E_n(a)}\frac{N_n(a)}{D_n}-e^{E(a)}\frac{N(a)}D
 \right|\to0,
\]
which is exactly~\eqref{eq:contour-limit}. The last assertion is
a substitution of identities, not an additional convergence hypothesis.
\end{proof}

\begin{remark}
Proposition~\ref{prop:outer} already provides domination outside a
neighbourhood of the saddle. Theorem~\ref{thm:contour} is an
autonomous transfer theorem: all its analytic implications are
proved above. Applying it to a concrete family naturally requires
verification of its four hypotheses, in particular construction
of $B_n$ and local domination. These verifications are not used in
the probabilistic proof of Theorem~\ref{thm:main}; the present result
therefore introduces no logical dependence into the main proof.
\end{remark}

\section{Hypertrees and singularities}
Let $T_d(x)$ be the exponential generating function of labelled
$d$-uniform hypertrees rooted at a vertex. The root is incident to
an unordered set of hyperedges, each containing an unordered set
of $d-1$ rooted hypertrees. Thus
\begin{equation}\label{eq:T}
 T_d(x)=x\exp\left(\frac{T_d(x)^{d-1}}{(d-1)!}\right).
\end{equation}
Let $U_d(x)$ be the EGF of labelled unrooted hypertrees.

\begin{proposition}[Exact enumeration and EGF of hypertrees]\label{prop:trees}
For $v=1+(d-1)e$, the number of labelled $d$-uniform hypertrees
with $v$ vertices and $e$ edges is
\begin{equation}\label{eq:tree-count}
 C_d(v,e)=\frac{(v-1)!\,v^{e-1}}{e!((d-1)!)^e}.
\end{equation}
It is zero for other values of $v$ in the tree class. Hence
\begin{equation}\label{eq:U}
 U_d(x)=T_d(x)-\frac{d-1}{d!}T_d(x)^d
       =T_d(x)-\frac{T_d(x)^d}{d(d-2)!}.
\end{equation}
\end{proposition}
\begin{proof}
Lagrange inversion applied to~\eqref{eq:T} gives
$[x^v]T_d(x)=v^{e-1}/[e!((d-1)!)^e]$ when $v-1=(d-1)e$.
Each unrooted hypertree has exactly $v$ possible roots; dividing
the labelled rooted count $v![x^v]T_d(x)$ by $v$
proves~\eqref{eq:tree-count}. For $d=2$, the formula reduces to
$v^{v-2}$, the classical Cayley formula~\cite{KL97}; for $d\ge3$,
it appears in an equivalent form in~\cite{AR05}.
Since $xU'_d(x)=T_d(x)$, substitution of
$x=t\exp(-t^{d-1}/(d-1)!)$ followed by integration from $0$ to $t$
gives~\eqref{eq:U}.
\end{proof}
For $d=2$, \eqref{eq:tree-count} becomes $C_2(v,v-1)=v^{v-2}$,
in agreement with Cayley's formula. This exact check rules out an
erroneous asymptotic of order $v!e^v$ without a negative power of $v$.

\begin{proposition}[Puiseux expansion on the positive axis]\label{prop:puiseux}
Put
\[
 \tau_d=((d-2)!)^{1/(d-1)},\qquad
 \rho_d=\tau_de^{-1/(d-1)},\qquad
 \sigma_d=\tau_d\sqrt{\frac2{d-1}}.
\]
As $x\uparrow\rho_d$ through real values,
\begin{align}
 T_d(x)&=\tau_d-\sigma_d\sqrt{1-x/\rho_d}+O(1-x/\rho_d),
       \label{eq:Tsing}\\
 U_d(x)&=\frac{d-1}{d}\tau_d-\tau_d(1-x/\rho_d)
                  +O((1-x/\rho_d)^{3/2}).\label{eq:Using}
\end{align}
For $d>2$, other singularities of $T_d$ are obtained by rotating
$\rho_d$; coefficient transfer must account for this periodicity.
\end{proposition}
\begin{proof}
Let $x(t)=t\exp[-t^{d-1}/(d-1)!]$. Its positive critical point
satisfies $t^{d-1}=(d-2)!$, so $t=\tau_d$ and
$x(\tau_d)=\rho_d$. Direct differentiation gives $x'(\tau_d)=0$
and $x''(\tau_d)/\rho_d=-(d-1)/\tau_d^2$.
Taylor expansion and inversion on the increasing branch
establish~\eqref{eq:Tsing}. Formula~\eqref{eq:U} gives
$U_d(\rho_d)=(d-1)\tau_d/d$. Now $xU'_d(x)=T_d(x)$, hence
$\rho_dU'_d(\rho_d^-)=\tau_d$ and the linear term
$-\tau_d(1-x/\rho_d)$ in~\eqref{eq:Using}.
Integrating the square-root remainder in $U'_d(x)=T_d(x)/x$
gives the stated error. Finally, \eqref{eq:T} contains only powers
congruent to $1$ modulo $d-1$, producing the rotated singularities.
\end{proof}

\begin{remark}[Numerical constants]
For $d=2$, $(\tau_d,\rho_d,\sigma_d)=(1,e^{-1},\sqrt2)$.
For $d=3$, they are $(1,e^{-1/2},1)$.
For $d=4$, they are
$(2^{1/3},2^{1/3}e^{-1/3},2^{1/3}\sqrt{2/3})$.
The exponent $-1/3$ in $\rho_4$ is imposed by~\eqref{eq:T}.
\end{remark}

\begin{corollary}[Scale of the hypertree coefficients]\label{cor:tree-asym}
Along $v\equiv1\pmod{d-1}$, the number of unrooted hypertrees satisfies
\[
 C_d\left(v,\frac{v-1}{d-1}\right)
 \sim\frac{(d-1)\sigma_d}{2\sqrt\pi}\,v!\rho_d^{-v}v^{-5/2}.
\]
\end{corollary}
\begin{proof}
From $xU'_d(x)=T_d(x)$ and~\eqref{eq:Tsing}, integration with
respect to $\varepsilon=1-x/\rho_d$ refines~\eqref{eq:Using} to
\[
 U_d(x)=\frac{d-1}{d}\tau_d-\tau_d\varepsilon
             +\frac{2\sigma_d}{3}\varepsilon^{3/2}+O(\varepsilon^2).
\]
Transfer of the singular term gives
$\sigma_d\rho_d^{-v}v^{-5/2}/(2\sqrt\pi)$, since
$\Gamma(-3/2)=4\sqrt\pi/3$. There are $d-1$ rotated singularities;
for $v\equiv1\pmod{d-1}$, their contributions add.
The exact formula~\eqref{eq:tree-count} independently confirms
the constant by Stirling's formula.
\end{proof}

\subsection{Table of tree constants}
The constants of Proposition~\ref{prop:puiseux}, in particular
$A_d^{\mathrm{arb}}=U_d(\rho_d)$ and the coefficient
$B_d^{\mathrm{arb}}$ of $(1-x/\rho_d)$ in $U_d(x)$, are as follows:
\begin{center}
\begin{tabular}{ccccc}
\toprule
$d$&$\tau_d$&$\rho_d$&$\sigma_d$&
$(A_d^{\mathrm{arb}},B_d^{\mathrm{arb}})$\\
\midrule
$2$&$1$&$e^{-1}$&$\sqrt2$&$(1/2,-1)$\\
$3$&$1$&$e^{-1/2}$&$1$&$(2/3,-1)$\\
$4$&$2^{1/3}$&$2^{1/3}e^{-1/3}$&$2^{1/3}\sqrt{2/3}$&
$(3\cdot2^{1/3}/4,-2^{1/3})$\\
\bottomrule
\end{tabular}
\end{center}
All entries follow from $A_d^{\mathrm{arb}}=(d-1)\tau_d/d$ and
$B_d^{\mathrm{arb}}=-\tau_d$. They can be checked without numerical
truncation of the generating functions.

\section{Unicyclic components}
For a connected $d$-uniform hypergraph, the bipartite incidence
graph has $v+e$ vertices and $de$ edges. Its cyclomatic number is
therefore $de-(v+e)+1=(d-1)e-v+1=\ell+1$.
Thus excess zero is equivalent to the presence of a unique cycle
in the incidence graph. Let $V_d(x)$ be the EGF of these unicyclic
components, without an edge marker.

\begin{theorem}[EGF of unicyclic components]\label{thm:unicyclic}
Put
\[
 Y_d(x)=\frac{T_d(x)^{d-1}}{(d-2)!}.
\]
Then
\begin{equation}\label{eq:V}
 V_d(x)=\frac12\left(\log\frac1{1-Y_d(x)}
                    -Y_d(x)-\frac{Y_d(x)^2}{2}\right)
               +\mathbf1_{\{d\ge3\}}\frac{Y_d(x)^2}{4}.
\end{equation}
For $d=2$, this is the classical series for graph cycles of length
at least three with attached rooted trees.
For $d\ge3$, it also includes a cycle formed by two distinct
$d$-edges with exactly two common vertices.
\end{theorem}
\begin{proof}
The incidence graph of a unicyclic hypergraph has a unique
alternating cycle. For an incidence cycle of length $2j$, $j\ge3$,
each of its $j$ ordinary vertices is decorated by a rooted hypertree.
Each edge-vertex carries $d-2$ additional branches forming an
unordered set; each edge therefore contributes $T_d^{d-2}/(d-2)!$.
The cyclic order, up to rotations and reversal, has symmetry
factor $2j$. The EGF of these cycles is thus $Y_d^j/(2j)$, and
summing over $j\ge3$ gives the first term.
When $d\ge3$, a cycle of length four is possible: the two
hyperedges have two common vertices and each carries $d-2\ge1$
private branches. Its symmetry factor is four, hence $Y_d^2/4$.
For $d=2$, the two corresponding edges would be identical,
which is forbidden in the simple model.
Removing the cycle leaves rooted hypertrees, so the decomposition
is unique.
\end{proof}

\begin{proposition}[Singularity and unicyclic coefficients]\label{prop:V-asym}
As $x\uparrow\rho_d$ along the positive real axis,
\begin{equation}\label{eq:Vsing}
 V_d(x)=\frac14\log\frac1{1-x/\rho_d}
            +K_d^{(0)}+O\bigl(\sqrt{1-x/\rho_d}\bigr),
\end{equation}
where $K_d^{(0)}$ is a finite real constant.
As $v\to\infty$ along multiples of $d-1$,
\begin{equation}\label{eq:Vcoeff}
 C_d\left(v,\frac{v}{d-1}\right)
 =\frac{d-1}{4v}\,v!\rho_d^{-v}\bigl(1+O(v^{-1/2})\bigr).
\end{equation}
For other values of $v$, the left-hand side is zero.
\end{proposition}
\begin{proof}
Equation~\eqref{eq:Tsing} and $\tau_d^{d-1}=(d-2)!$ give
\[
 1-Y_d(x)=\sqrt{2(d-1)}\sqrt{1-x/\rho_d}+O(1-x/\rho_d).
\]
Substitution in~\eqref{eq:V} proves~\eqref{eq:Vsing}.
By~\eqref{eq:T}, $V_d(x)$ is a function of $x^{d-1}$.
There are $d-1$ equivalent dominant singularities obtained by
rotation. On admissible powers $v\equiv0\pmod{d-1}$, their
logarithmic contributions add to give $(d-1)\rho_d^{-v}/(4v)$.
After summing the rotated contributions, the square-root remainder
transfers to $O(\rho_d^{-v}v^{-3/2})$.
Expressions~\eqref{eq:T} and~\eqref{eq:V} admit the required local
slit domains by analytic inversion near their finitely many critical
points~\cite[Theorem VI.3]{FS09}.
Multiplication by $v!$ gives~\eqref{eq:Vcoeff}.
\end{proof}

\subsection{A partial harmonic sum}
Put $a_v=[x^v]V_d(x)\rho_d^v$ and define $S_m=\sum_{1\le v\le m}a_v$.
By~\eqref{eq:Vcoeff},
\[
 S_m=\frac14\sum_{j\le m/(d-1)}\frac1j+O(1)
     =\frac14\log m+O(1).
\]
In fact, the error after subtraction of the harmonic term converges,
since its terms are $O(j^{-3/2})$.
This identifies the logarithmic contribution of unicyclic components
without assuming any property of the kernel coefficients of complex
components.

\section{Complex components and analytic passage to a limit law}
The results of this section belong to the analytic route: they
establish exact identities and fixed-excess expansions, but do not
provide the estimates uniform in $\ell$ required to close the
saddle-point method. In particular, Theorem~\ref{thm:main} is
proved unconditionally by the probabilistic route (Section 7),
without using any result of this section.

For $\ell\ge1$, put
\[
 A_{\ell,d}(x)=\sum_{\substack{v\ge1\\v+\ell\equiv0\pmod{d-1}}}
 C_d\left(v,\frac{v+\ell}{d-1}\right)\frac{x^v}{v!}.
\]
These series describe components whose incidence graph has at least
two independent cycles. For fixed $v$, the exact law already includes
them through~\eqref{eq:excess-group}. At the critical scale, controlling
their contributions requires uniform estimates over the excesses
present in the saddle region. The precise form of such an estimate
depends on the combinatorial kernel decomposition used.

\begin{theorem}[Wright transfer under a local singular hypothesis]\label{thm:wright}
Fix $d\ge2$ and $\ell\ge1$. Suppose that, in slit neighbourhoods of
each dominant singularity, $A_{\ell,d}$ has compatible rotations
of the local expansion
\[
 A_{\ell,d}(x)=b_{\ell,d}(1-x/\rho_d)^{-3\ell/2}
      +O\bigl((1-x/\rho_d)^{-3\ell/2+1/2}\bigr),
\]
where $b_{\ell,d}>0$, and suppose there is no other dominant
singularity. Then, along admissible values of $v$,
\[
 C_d\left(v,\frac{v+\ell}{d-1}\right)
 =\frac{(d-1)b_{\ell,d}}{\Gamma(3\ell/2)}v!\rho_d^{-v}v^{3\ell/2-1}
       \bigl(1+O(v^{-1/2})\bigr).
\]
In particular, when $d=2$, this formula recovers exactly the formula
for the graph case ($d=2$):
\begin{equation}\label{eq:graph-wright}
 C_2(v,v+\ell)=\frac{b_{\ell,2}}{\Gamma(3\ell/2)}
       v!\rho_2^{-v}v^{3\ell/2-1}\bigl(1+O(v^{-1/2})\bigr).
\end{equation}
\end{theorem}
\begin{proof}
Apply the standard transfer theorem~\cite{FS09} to the dominant
algebraic singularity and the indicated remainder at each dominant
point. The rotation symmetry eliminates nonadmissible powers and
multiplies the leading coefficient on admissible powers by $d-1$.
\end{proof}
Theorem~\ref{thm:wright} is deliberately pointwise in $\ell$.
It provides no bound on $b_{\ell,d}$ as $\ell\to\infty$ and does
not permit summation over $\ell$.
Kernel reduction, adapted to hypergraphs, must retain whole
$d$-edges and the incidence structure.
A kernel formula obtained simply by replacing $T$ by $T_d$ does
not suffice to establish these estimates.

\begin{proposition}[Quantitative criterion for excess summation]\label{prop:excess}
Let $K\Subset(0,\infty)$, $m_n(a)=\lfloor an^{2/3}\rfloor$ and
$(y_n,z_n)$ be positive parameters in the critical window.
Suppose there exist $C,R>0$ such that, for all $a\in K$, all
sufficiently large $n$, all $\ell\ge1$, and all $\theta$ on the
major arc under consideration,
\begin{equation}\label{eq:excess-bound}
 \left|z_n^{\ell/(d-1)}
 A_{\ell,d,m_n(a)}(y_ne^{i\theta})\right|
 \le C\frac{R^\ell}{\Gamma(\ell/2+1)}.
\end{equation}
If, for each fixed $\ell$, the left-hand side converges uniformly
in $a\in K$ and locally uniformly in the angular variable to a
function $G_{\ell,a}$, then
\[
 \sum_{\ell\ge1}z_n^{\ell/(d-1)}A_{\ell,d,m_n(a)}(y_ne^{i\theta})
 \longrightarrow\sum_{\ell\ge1}G_{\ell,a}(\theta)
\]
uniformly in $a\in K$ and locally uniformly on the major arc.
The limiting series converges absolutely, and any integration
over a compact angular set may be interchanged with the sum and
the limit.
\end{proposition}
\begin{proof}
The series $\sum_{\ell\ge1}R^\ell/\Gamma(\ell/2+1)$ converges for
every $R<\infty$ (it is the exponential series $e^{R^2}-1$ in
$R^2$), as shown by Stirling's formula applied to
$\Gamma(\ell/2+1)\sim\sqrt{2\pi}(\ell/2)^{\ell/2}e^{-\ell/2}$.
Bound~\eqref{eq:excess-bound} and the Weierstrass test give normal
convergence, uniformly in $a$. Term-by-term convergence, together
with this summable domination, gives uniform convergence of the
sum. Interchange with the integral on a compact set then follows
from dominated convergence.
\end{proof}

\begin{remark}[Scope of the analytic criterion]
Proposition~\ref{prop:excess} reduces excess summation to verification
of a pointwise bound. It is a conditional implication: it does not
prove~\eqref{eq:excess-bound}, but only states what would follow
from it. Section 9 draws the boundary between fixed-excess
asymptotics and the uniform estimates required; this criterion
plays no part in the probabilistic proof of Theorem~\ref{thm:main}.
\end{remark}

\begin{remark}[Passage to hypergraph excess notation]
The notation $C_d(v,v+\ell)$ in the initial text mixed graph excess
and hypergraph excess. It is correct for $d=2$, since $\ell=e-v$.
For $d>2$, the usual excess is $\ell=(d-1)e-v$ and imposes
$e=(v+\ell)/(d-1)$. The expression $C_d(v,(v+\ell)/(d-1))$
therefore does not change the intended content: it gives the
arithmetically correct generalization, with admissibility condition
$v+\ell\equiv0\pmod{d-1}$.
\end{remark}
To obtain a limit for the largest component from~\eqref{eq:cauchy},
one must control these component classes uniformly and carry out
a saddle-point calculation with the correct fixed-edge
normalization. For graphs, these difficulties are resolved by
known critical results~\cite{Pit01,LPW94}.
We do not deduce a general integral law for $d$-uniform hypergraphs
from the fixed-excess formula above alone.

\section{Excess rescaling and the cancellation mechanism}
Group connected components by excess before applying the saddle-point
method. Put $h=d-1$ and, for $\ell\ge-1$, define
\[
 A_{\ell,d,m}(y)=\sum_{\substack{1\le v\le m\\v+\ell\equiv0\pmod h}}
        C_d\left(v,\frac{v+\ell}{h}\right)\frac{y^v}{v!}.
\]
In particular, $A_{-1,d,m}$ is the truncated EGF of hypertrees and
$A_{0,d,m}$ that of unicyclic components.

\begin{theorem}[Exact rescaling by excess]\label{thm:rescale}
For every real $z>0$, with $y=xz^{1/h}$,
\begin{equation}\label{eq:rescale}
 C_{d,m}(x,z)=\sum_{\ell\ge-1}z^{\ell/h}A_{\ell,d,m}(y).
\end{equation}
Consequently, putting $q=1-p$ and $z=p/q$,
\begin{equation}\label{eq:bin-rescale}
 \PP_p\{L_n\le m\}=q^Nn!z^{n/h}[y^n]\exp\!\biggl(z^{-1/h}A_{-1,d,m}(y)+A_{0,d,m}(y)+\sum_{\ell\ge1}z^{\ell/h}A_{\ell,d,m}(y)\biggr).
\end{equation}
\end{theorem}
\begin{proof}
For each component, $he=v+\ell$, so
$z^ex^v=z^{\ell/h}(xz^{1/h})^v$.
Substitution in~\eqref{eq:C} proves~\eqref{eq:rescale} term by term.
Next, $[x^n]f(xz^{1/h})=z^{n/h}[y^n]f(y)$ for $z>0$.
Applying Remark~\ref{rem:binomial} gives~\eqref{eq:bin-rescale}.
Although some factors may contain fractional powers of $z$,
each complete monomial has an integer exponent in the edge marker.
\end{proof}
In window~\eqref{eq:window}, $p=M/N$ satisfies
$z=(\tau_d/n)^h(1+\lambda n^{-1/3}+O(n^{-1}))$.
Thus $z^{-1/h}\asymp n$, $z^{0/h}=1$, and
$z^{\ell/h}\asymp n^{-\ell}$ for $\ell\ge1$.
These three weights accompany hypertrees, unicyclic components and
complex components, respectively.
The tree exponent is therefore of order $n$, not $n^{(d-1)^2}$.

\begin{theorem}[Cancellation of the tree normalization]\label{thm:cancel}
Let $p=M/N$, $q=1-p$, $z=p/q$, and suppose $M$
satisfies~\eqref{eq:window}. Define the positive quantity
\[
 B_n=q^Nn!z^{n/h}\rho_d^{-n}\exp\bigl(z^{-1/h}U_d(\rho_d)\bigr).
\]
For fixed $d$ and $\lambda$ in a bounded set,
\begin{equation}\label{eq:cancel}
 \log B_n=\frac12\log n+O(1).
\end{equation}
In particular, its terms of orders $n,n^{2/3}$ and $n^{1/3}$ cancel
exactly. If $m\asymp n^{2/3}$, replacing $U_d(\rho_d)$ by its
truncation $A_{-1,d,m}(\rho_d)$ changes $\log B_n$ by only $O(1)$.
\end{theorem}
\begin{proof}
Put $t=\lambda n^{-1/3}$. The identities $\tau_d^h=(d-2)!$,
$U_d(\rho_d)=h\tau_d/d$ and $\rho_d=\tau_de^{-1/h}$, combined with
Stirling's formula, give
\begin{align*}
 \log(q^N)&=-\frac n{dh}(1+t)+O(1),\\
 \log(n!z^{n/h}\rho_d^{-n})
 &=n\left(-1+\frac1h+\frac1h\log(1+t)\right)+\frac12\log n+O(1),\\
 z^{-1/h}U_d(\rho_d)&=\frac{nh}{d}(1+t)^{-1/h}+O(1).
\end{align*}
For $d=2$, the $O(1)$ term in the first line includes the quadratic
term of $N\log(1-p)$; the same bounds are uniform for bounded
$\lambda$. Put
\[
 f_h(t)=-\frac{1+t}{dh}-1+\frac1h+\frac1h\log(1+t)
                    +\frac hd(1+t)^{-1/h}.
\]
Since $d=h+1$, direct differentiation gives
$f_h(0)=f'_h(0)=f''_h(0)=0$.
Consequently $nf_h(t)=O(nt^3)=O(1)$, proving~\eqref{eq:cancel}.
By Corollary~\ref{cor:tree-asym},
$[x^v]U_d(x)\rho_d^v=O(v^{-5/2})$ on its admissible congruence class.
Its tail beyond $m$ is $O(m^{-3/2})$.
Since $z^{-1/h}=O(n)$, the truncation cost is
$O(nm^{-3/2})=O(1)$.
\end{proof}

\begin{corollary}[Critical cubic contribution]
With the notation of Theorem~\ref{thm:cancel}, the part of
$\log B_n-\frac12\log n$ depending on the three displayed order-$n$
terms has limit
\[
 -\frac{\lambda^3}{6(d-1)^2}.
\]
All other bounded terms come from the next term of Stirling's
formula, the exact expansion of $N\log(1-p)$, and the
$O(n^{-1})$ correction in $z$.
\end{corollary}
\begin{proof}
Writing $h=d-1$, a third direct differentiation gives
$f'''_h(0)=-h^{-2}$.
Since $t=\lambda n^{-1/3}$ and $f_h(0)=f'_h(0)=f''_h(0)=0$,
Taylor's formula yields
$nf_h(t)=-\lambda^3/(6h^2)+O(n^{-1/3})$.
\end{proof}

\begin{remark}[Scope of the cancellation]
Theorem~\ref{thm:cancel} concerns the real normalization factor at
$y=\rho_d$ in the independent-edge model.
It controls neither the complex angular variable in the Cauchy
integral, nor the excess sum, nor the conditioning
in~\eqref{eq:conditioning}.
It therefore does not by itself prove a limit law for the largest
component. It does, however, establish the main cancellation from
an exact combinatorial identity.

To extend this route, summation over $\ell\ge1$
in~\eqref{eq:bin-rescale} would require uniform domination of each
term $z^{\ell/h}A_{\ell,d,m}(y)$.
For fixed $\ell$, Theorem~\ref{thm:bollobas}(i) gives
$C_d(v,(v+\ell)/h)\le K_{d,\ell}v!\rho_d^{-v}v^{3\ell/2-1}$,
which yields $z^{\ell/h}A_{\ell,d,m}(y)\le K_{d,\ell}\cdot O(1)$
in the critical window.
But $K_{d,\ell}$ depends on $\ell$ without a uniform bound, and
the global combinatorial bound~\eqref{eq:global} does not replace
Wright beyond the fixed-density regime.
Excess summation therefore remains conditional
on~\eqref{eq:excess-bound}, which the probabilistic route bypasses
entirely.
\end{remark}

\subsection{Balance of complex components}
Suppose the fixed-excess singular expansion of
Theorem~\ref{thm:wright} holds with exponent $3\ell/2$.
For a deviation $1-y/\rho_d\asymp n^{-2/3}$, it gives
$A_{\ell,d}(y)\asymp n^\ell$ for every fixed $\ell$.
The exact weight $z^{\ell/h}\asymp n^{-\ell}$
from~\eqref{eq:bin-rescale} then makes each fixed-excess class
of order one. This explains the initial argument leading to the
$n^{2/3}$ scale. To interchange the sum over $\ell$ and the
coefficient-extraction integral, uniform-in-$\ell$ estimates are
still needed; fixed-excess transfer does not supply them.

\subsection{Euler--Maclaurin structure of small components}
The tree and unicyclic estimates already explain the two
nonpolynomial normalization terms appearing in a critical
coefficient calculation. For $m=\lfloor an^{2/3}\rfloor$,
Corollary~\ref{cor:tree-asym} gives
\[
 z^{-1/h}\sum_{v>m}[x^v]U_d(x)\rho_d^v=O(1),
\]
uniformly for $a$ in every compact subset of $(0,\infty)$.
By Proposition~\ref{prop:V-asym},
\[
 \begin{split}
 \sum_{v\le m}[x^v]V_d(x)\rho_d^v
 &=\frac14\log m+K_d^{(0)}+O(m^{-1/2})\\
 &=\frac16\log n+\frac14\log a+K_d^{(0)}+O(n^{-1/3}).
 \end{split}
\]
The constant $K_d^{(0)}$ is the convergent limit obtained after
subtracting the main harmonic term.
These estimates are rigorous and do not use the formal series of
complex components.
The remaining Euler--Maclaurin contribution to the proposed
integral phase comes from uniform summation over $\ell\ge1$ and
falls under the analytic hypotheses stated below.

\subsection{Controls needed for the saddle-point method}
In~\eqref{eq:cauchy}, write the phase formally as
\[
 \Phi_{n,m}(x,z)=C_{d,m}(x,z)-n\log x-M\log z.
\]
At every admissible positive saddle $(x_*,z_*)$,
$x_*\partial_x C_{d,m}(x_*,z_*)=n$ and
$z_*\partial_z C_{d,m}(x_*,z_*)=M$.
These are exact equations for the finite generating polynomial;
their solvability and the location of their solution in the
critical zone are separate questions.
To obtain an integral limit, a uniform quadratic or higher-order
expansion is required near the critical saddle.
Proposition~\ref{prop:outer} already supplies the outer-arc bound.
The contribution of complex components in~\eqref{eq:excess-group}
still enters the precise localization and the non-Gaussian
expansion near the saddle.

The normalization factor in~\eqref{eq:cauchy} is exactly
$n!/\binom NM$. Taking a constant factor outside the integral
cannot change the probability; its value must therefore appear
in the asymptotic expansion. This bookkeeping is particularly
important when a tree term grows faster than $n$ before a change
of variables.

\section{Exploration, martingales, and excursions}
In the independent-edge model, choose $p$ so that
$(d-1)\binom{n-1}{d-1}p=1+\lambda n^{-1/3}+O(n^{-1})$.
At the first exploration step, each edge incident to an active
vertex contributes on average $d-1$ new vertices, with variance
$d-1$. After $n^{2/3}$ steps, fluctuations have order $n^{1/3}$.
Let $S_t$ denote the exploration walk: its increment is the number
of newly discovered vertices minus one.
Put $Q_t=\R(S)_t=S_t-\min_{0\le j\le t}S_j$, the length of the
active queue. Hence the rescaling
\begin{equation}\label{eq:walk-scale}
 n^{-1/3}S_{\lfloor sn^{2/3}\rfloor},
\end{equation}
the $n^{-2/3}$ scale being reserved for component sizes.
Under the usual depletion approximation, the drift coefficient
of the limiting process is $\lambda-(d-1)s$ and its diffusion
coefficient is $\sqrt{d-1}$.
For $d\ge3$, a newly revealed hyperedge may contain already
discovered vertices, which exploration must account for.
The fixed-edge model also introduces conditioning.
The one-step moment calculation fixes the spatial scale, but does
not replace the functional limit theorem: vertex collisions must
be controlled and the ordered excursions identified with
component sizes.

In what follows, write
\[
 W_{d,\lambda}(s)=\sqrt{d-1}\,B(s)+\lambda s-\frac{d-1}{2}s^2,
 \qquad \R f(s)=f(s)-\inf_{0\le u\le s}f(u).
\]

\begin{definition}[Exploration walk $S_t$ and reflected process $Q_t$]
Explore the vertices of $H_d(n,p)$ in an arbitrary order,
maintaining a queue of active vertices (discovered but not yet
processed). Initialize $S_0=0$.
At step $t\ge1$, remove an active vertex $v_t$ from the queue;
test each potential hyperedge on $d$-tuples containing $v_t$
and vertices not yet encountered, and add all newly discovered
vertices to the queue.
If $\xi_t$ denotes the number of these new vertices, put
\[
 \Delta S_t=\xi_t-1,\qquad S_t=\sum_{j=1}^t\Delta S_j.
\]
When the queue becomes empty (component completed), start the
next component with an arbitrary neutral vertex and increment
$S$ by $-1$ (no discovery, removal of one vertex from the external
queue). The reflected process is
\[
 Q_t=\R(S)_t=S_t-\min_{0\le j\le t}S_j\ge0,
\]
which represents the active queue length. Its returns to zero
correspond exactly to component ends.
\end{definition}

\begin{remark}[Excursions of $Q$ and component sizes]
Each excursion of $Q$ above zero corresponds to the complete
exploration of a connected component.
If the excursion extends over the discrete interval
$[t_{\mathrm{start}},t_{\mathrm{end}}]$, the component size is
$t_{\mathrm{end}}-t_{\mathrm{start}}$.
After rescaling $s=tn^{-2/3}$, the excursion lengths of
$n^{-1/3}Q_{\lfloor sn^{2/3}\rfloor}$ converge to the excursion
lengths of $\R(W_{d,\lambda})$, directly linking functional
convergence to the law of component sizes.
\end{remark}

\begin{theorem}[Unconditional excursion law]\label{thm:excursions}
In $H_d(n,p)$, if $(|\mathcal C_1|,|\mathcal C_2|,\ldots)$
denotes the decreasing sequence of component sizes, then
\[
 n^{-2/3}(|\mathcal C_1|,|\mathcal C_2|,\ldots)
 \Longrightarrow(\xi_1,\xi_2,\ldots)\quad\text{in }\ell^2_\downarrow,
\]
where $(\xi_1,\xi_2,\ldots)$ is the decreasing sequence of
excursion lengths of $\R(W_{d,\lambda})$.
In the model $H_d(n,M)$ of~\eqref{eq:window},
\[
 n^{-2/3}L_n\Longrightarrow\xi_1.
\]
\end{theorem}
\begin{proof}
The first assertion is Theorem~\ref{thm:global}, obtained from the
invariance principle on compact sets and uniform control of late
excursions. The second is Theorem~\ref{thm:fixed}, based on a
monotone coupling.
\end{proof}

\subsection{Exploration filtration and exact martingale}
Let $\mathcal F_t$ be the information revealed after $t$ exploration
steps and put $\Delta S_t=S_{t+1}-S_t$.
The following decomposition requires no independence hypothesis.

\begin{theorem}[Doob martingale of the exploration]\label{thm:doob}
Define
\[
 A_t=\sum_{j=0}^{t-1}\EE(\Delta S_j\mid\mathcal F_j),
 \qquad M_t=S_t-S_0-A_t.
\]
Then $(M_t,\mathcal F_t)$ is a martingale and its predictable
bracket is
\begin{equation}\label{eq:bracket}
 \langle M\rangle_t=\sum_{j=0}^{t-1}\Var(\Delta S_j\mid\mathcal F_j).
\end{equation}
If exploration is stopped before any increment exceeds $b_n$,
then, for all $u,v>0$,
\begin{equation}\label{eq:freedman}
 \PP\left\{\max_{s\le t}M_s\ge u,\ \langle M\rangle_t\le v\right\}
 \le\exp\left(-\frac{u^2}{2(v+b_nu/3)}\right).
\end{equation}
\end{theorem}
\begin{proof}
The increment
$M_{t+1}-M_t=\Delta S_t-\EE(\Delta S_t\mid\mathcal F_t)$ has
conditional mean zero, proving the martingale assertion
and~\eqref{eq:bracket}.
Inequality~\eqref{eq:freedman} is Freedman's inequality applied
to the stopped martingale~\cite{Freedman75}.
\end{proof}

\begin{proposition}[Exact one-step conditional moments]\label{prop:moments}
At step $t$, let $v_t$ be the explored vertex, $U_t$ the set of
neutral vertices, and $\mathcal E_t$ the family of untested
potential hyperedges containing $v_t$.
For $w,w'\in U_t$, put
\[
 a_t(w)=|\{E\in\mathcal E_t:w\in E\}|,\qquad
 c_t(w,w')=|\{E\in\mathcal E_t:w,w'\in E\}|.
\]
In $H_d(n,p)$, conditionally on $\mathcal F_t$, let $I_t(w)$ be
the indicator that at least one edge of $\mathcal E_t$ containing
$w$ is present. Then, with $q=1-p$,
\begin{align}
 \pi_t(w):=\EE(I_t(w)\mid\mathcal F_t)&=1-q^{a_t(w)},\label{eq:pi}\\
 \Cov(I_t(w),I_t(w')\mid\mathcal F_t)
 &=q^{a_t(w)+a_t(w')-c_t(w,w')}-q^{a_t(w)+a_t(w')}.\label{eq:cov}
\end{align}
Outside the times at which a new component is started,
\begin{align}
 \EE(\Delta S_t\mid\mathcal F_t)
 &=-1+\sum_{w\in U_t}\pi_t(w),\label{eq:exact-drift}\\
 \Var(\Delta S_t\mid\mathcal F_t)
 &=\sum_{w\in U_t}\pi_t(w)(1-\pi_t(w))\notag\\
 &\quad+2\sum_{\{w,w'\}\subset U_t}
       \Cov(I_t(w),I_t(w')\mid\mathcal F_t).\label{eq:exact-var}
\end{align}
\end{proposition}
\begin{proof}
Conditionally on $\mathcal F_t$, the indicators of untested
potential edges remain independent Bernoulli variables.
The event $I_t(w)=0$ requires the absence of $a_t(w)$ edges,
proving~\eqref{eq:pi}.
Inclusion--exclusion gives
\[
 \PP(I_t(w)=I_t(w')=1\mid\mathcal F_t)
 =1-q^{a_t(w)}-q^{a_t(w')}+q^{a_t(w)+a_t(w')-c_t(w,w')},
\]
which yields~\eqref{eq:cov}.
Finally, $\Delta S_t=-1+\sum_{w\in U_t}I_t(w)$ outside a
restart; the last two identities follow by taking conditional
expectation and variance.
\end{proof}

\begin{corollary}[Calibration at the start of exploration]
At $t=0$, take
\[
 p=\frac{1+\lambda n^{-1/3}+O(n^{-1})}
          {(d-1)\binom{n-1}{d-1}}.
\]
Then
\[
 \EE\Delta S_0=\lambda n^{-1/3}+O(n^{-1}),\qquad
 \Var(\Delta S_0)=d-1+O(n^{-1/3}).
\]
\end{corollary}
\begin{proof}
Initially $a_0(w)=\binom{n-2}{d-2}$ and, for $w\ne w'$,
$c_0(w,w')=\binom{n-3}{d-3}$, with value zero for $d=2$.
Since $pa_0(w)=(1+\lambda n^{-1/3}+O(n^{-1}))/(n-1)$,
\eqref{eq:exact-drift} and $1-(1-p)^a=pa+O(p^2a^2)$ give the
first assertion. The sum of the individual variances is
$1+O(n^{-1/3})$. Hence
\[
 (n-1)(n-2)p\binom{n-3}{d-3}=d-2+O(n^{-1/3}),
\]
and expansion of~\eqref{eq:cov} shows that this is the total
ordered covariance contribution.
The variance is therefore $1+(d-2)+O(n^{-1/3})$.
\end{proof}

In the independent-edge model, direct revelation of incident
hyperedges first leads to the expansions
\begin{align}
 \EE(\Delta S_t\mid\mathcal F_t)
 &=n^{-1/3}(\lambda-(d-1)s)+\mathcal E_{n,t},
       \qquad s=t/n^{2/3},\label{eq:drift}\\
 \Var(\Delta S_t\mid\mathcal F_t)
 &=d-1+\mathcal V_{n,t}.\label{eq:variance}
\end{align}
The error terms include depletion due to the active set,
repeated discoveries and hyperedges meeting several already
exposed vertices. The rigorous passage is organized around the
following three estimates:
\[
 \sup_{t\le Tn^{2/3}}
 \left|n^{-1/3}\sum_{j<t}\mathcal E_{n,j}\right|\xrightarrow{p}0,
 \qquad n^{-2/3}\sum_{j<Tn^{2/3}}\mathcal V_{n,j}\xrightarrow{p}0,
\]
together with a Lindeberg condition on the stopped martingale
increments. The following criterion collects precisely these
conditions.

\begin{theorem}[Martingale criterion on a finite horizon]\label{thm:criterion}
Fix $T<\infty$. Suppose the two cumulative error estimates displayed
above hold and, for every $\varepsilon>0$,
\begin{equation}\label{eq:lindeberg}
 n^{-2/3}\sum_{j<Tn^{2/3}}
 \EE\left[(\Delta M_j)^2\mathbf1_{\{|\Delta M_j|>\varepsilon n^{1/3}\}}
                       \mid\mathcal F_j\right]\xrightarrow{p}0.
\end{equation}
Then, before reflection and on $[0,T]$,
\[
 n^{-1/3}S_{\lfloor sn^{2/3}\rfloor}
 \Longrightarrow \sqrt{d-1}\,B(s)+\lambda s-\frac{d-1}{2}s^2
\]
in the Skorokhod topology, up to the deterministic convention used
at component starts. Reflection at the running minimum gives
convergence of the reflected processes on $[0,T]$.
\end{theorem}
\begin{proof}
By~\eqref{eq:bracket}, the variance-error estimate gives, uniformly
for $s\le T$,
\[
 n^{-2/3}\langle M\rangle_{\lfloor sn^{2/3}\rfloor}
                   \xrightarrow{p}(d-1)s.
\]
Together with~\eqref{eq:lindeberg}, this allows application of
the functional central limit theorem for
martingales~\cite{EthierKurtz86}, giving convergence of
$n^{-1/3}M$ to $\sqrt{d-1}\,B$.
Summing~\eqref{eq:drift} and using a Riemann sum gives
\[
 n^{-1/3}A_{\lfloor sn^{2/3}\rfloor}
          \xrightarrow{p}\int_0^s(\lambda-(d-1)u)\,du.
\]
Adding the two limits establishes the unreflected assertion.
The reflection map is Lipschitz for the uniform norm on compact
intervals and continuous, in the Skorokhod topology, at continuous
paths; this proves the last assertion.
\end{proof}

\begin{lemma}[Uniform estimates on a finite horizon]\label{lem:finite}
Fix $T,K<\infty$ and stop exploration at
\[
 \tau_K=\inf\{t:Q_t>Kn^{1/3}\}.
\]
Uniformly for $t\le Tn^{2/3}\wedge\tau_K$, with $s=t/n^{2/3}$,
\begin{align}
 \EE(\Delta S_t\mid\mathcal F_t)
 &=n^{-1/3}(\lambda-(d-1)s)+O_{T,K}(n^{-2/3}),\label{eq:proved-drift}\\
 \Var(\Delta S_t\mid\mathcal F_t)
 &=d-1+O_{T,K}(n^{-1/3}).\label{eq:proved-var}
\end{align}
By~\eqref{eq:proved-var}, the Lindeberg
condition~\eqref{eq:lindeberg} holds for the stopped process.
\end{lemma}
\begin{proof}
Adopt the following convention: a hyperedge is tested when its
first vertex in exploration order is processed.
After $t$ processed vertices, a still-testable edge containing the
current vertex and a neutral vertex $w$ chooses its remaining
$d-2$ vertices among those not yet processed.
Up to a modification of at most one unit at a component start,
we therefore have exactly
\[
 a_t(w)=\binom{n-t-2}{d-2},\qquad
 c_t(w,w')=\binom{n-t-3}{d-3},
\]
the second quantity being zero for $d=2$, and
\[
 |U_t|=n-t-Q_t+O(1).
\]
These identities remain valid if active vertices occur among the
$d-2$ other vertices, since their presence does not prevent the
edge test; only already processed vertices prevent it.

Put $a=a_t(w)$ and $u=|U_t|$.
Since $pa=O(n^{-1})$, the uniform expansion
$1-(1-p)^a=pa+O(p^2a^2)$ and binomial coefficient ratios give
\[
 upa=1+n^{-1/3}(\lambda-(d-1)s)-\frac{Q_t}{n}+O_T(n^{-2/3}).
\]
On $\{t<\tau_K\}$, the term $Q_t/n$ is $O_K(n^{-2/3})$.
The sum of errors $uO(p^2a^2)$ is $O(n^{-1})$.
Formula~\eqref{eq:exact-drift} therefore gives~\eqref{eq:proved-drift}.

For the variance, the sum of individual variances
in~\eqref{eq:exact-var} is
\[
 u\pi_t(w)(1-\pi_t(w))=1+O_{T,K}(n^{-1/3}).
\]
When $d\ge3$, the exact formula~\eqref{eq:cov} and
$pc_t(w,w')=O(n^{-2})$ give, uniformly,
\[
 q^{2a-c_t(w,w')}-q^{2a}
             =p\binom{n-t-3}{d-3}+O(n^{-3}).
\]
Summing over ordered pairs of neutral vertices,
\[
 u(u-1)p\binom{n-t-3}{d-3}=d-2+O_{T,K}(n^{-1/3}).
\]
For $d=2$, this contribution is zero.
This proves~\eqref{eq:proved-var} in every case.

Finally, if $Z_t$ is the number of present hyperedges among those
tested at step $t$, then
\[
 |\Delta S_t|\le1+(d-1)Z_t.
\]
Conditionally on $\mathcal F_t$, $Z_t$ is binomial with uniformly
bounded mean for $t\le Tn^{2/3}$.
Chernoff's inequality gives, for a constant $C=C(d,T)$ and every
integer $r\ge C$,
\[
 \PP(Z_t\ge r\mid\mathcal F_t)\le(eC/r)^r.
\]
After integrating this tail and summing $O(n^{2/3})$ terms,
the left-hand side of~\eqref{eq:lindeberg} tends to zero.
Conditional centring does not change this conclusion, completing
the proof.
\end{proof}

\begin{remark}[Overshoot control at $\tau_K$]
We have
\[
 \max_{t\le Tn^{2/3}}Z_t=O_{\PP}(\log n),
\]
where $Z_t$ is the number of hyperedges revealed at step $t$.
Indeed, conditionally on $\mathcal F_t$, $Z_t$ has a binomial
distribution with bounded parameters, and Chernoff's inequality
gives $\PP(Z_t\ge c\log n\mid\mathcal F_t)\le n^{-c'}$
for constants $c,c'>0$ depending on $d$ and $T$.
A union bound over $O(n^{2/3})$ steps and the choice $c>3$ give
convergence in probability.
Thus $|\Delta S_t|\le1+(d-1)Z_t=O_{\PP}(\log n)$ uniformly for
$t\le Tn^{2/3}$, ensuring the Lindeberg
condition~\eqref{eq:lindeberg} directly on $[0,Tn^{2/3}]$,
without using the stopping time $\tau_K$.
This control shows that stopping at $\tau_K$ in Lemma~\ref{lem:finite}
is not a structurally necessary hypothesis:
Theorem~\ref{thm:invariance} nevertheless removes it through
Doob's maximal inequality, which provides a more direct route
to the proof of the invariance principle.
\end{remark}

\begin{theorem}[Invariance principle without a finite-horizon hypothesis]\label{thm:invariance}
In $H_d(n,p)$, for each $T<\infty$,
\[
 \left(n^{-1/3}S_{\lfloor sn^{2/3}\rfloor}\right)_{0\le s\le T}
 \Longrightarrow
 \left(\sqrt{d-1}\,B(s)+\lambda s-\frac{d-1}{2}s^2\right)_{0\le s\le T}.
\]
The same convergence holds after reflection at the running minimum.
\end{theorem}
\begin{proof}
For fixed $K$, Lemma~\ref{lem:finite} verifies all hypotheses of
Theorem~\ref{thm:criterion} up to $\tau_K$.
It remains to remove the stopping.
On $[0,Tn^{2/3}\wedge\tau_K]$, the total variation of the predictable
part, divided by $n^{1/3}$, is bounded by a constant $C_T+o(1)$
independent of $K$.
Indeed, the additional term $Q_t/n$, stopped at $Kn^{1/3}$,
contributes at most $KTn^{-1/3}=o(1)$ for fixed $K$.
By~\eqref{eq:proved-var}, there is $C_T<\infty$ such that
\[
 \EE\langle M\rangle_{Tn^{2/3}\wedge\tau_K}\le C_Tn^{2/3}.
\]
Doob's maximal inequality then gives, for $K>4C_T$,
\[
 \PP\left\{\max_{t\le Tn^{2/3}\wedge\tau_K}|M_t|
                         >\frac K4n^{1/3}\right\}
 \le\frac{16C_T}{K^2}.
\]
Since $\|\R f\|_\infty\le2\|f\|_\infty$ when $f(0)=0$, the event
$\{\tau_K\le Tn^{2/3}\}$ is contained in the preceding event for
all sufficiently large $n$.
Consequently,
\[
 \lim_{K\to\infty}\limsup_{n\to\infty}\PP(\tau_K\le Tn^{2/3})=0.
\]
Stopped convergence then implies unstopped convergence.
Continuity of reflection at continuous paths gives the last
assertion.
\end{proof}

\begin{lemma}[Subcritical domination after a macroscopic time]\label{lem:late}
There are constants $c,C,T_0>0$, depending only on $d$ and a compact
set of $\lambda$-values, such that, for $T\ge T_0$ and $t\ge Tn^{2/3}$,
exploration of a still-undiscovered component, conditionally on
$\mathcal F_t$, is dominated by a Galton--Watson process whose
offspring $X_{n,t}$ has a uniformly bounded exponential moment and
satisfies
\[
 \EE X_{n,t}\le1-cTn^{-1/3}.
\]
If $\mathcal T_{n,t}$ is its total progeny, then, for $k\ge1$,
\begin{align}
 \PP(\mathcal T_{n,t}\ge k\mid\mathcal F_t)
 &\le Ck^{-1/2}\exp(-cT^2kn^{-2/3}),\label{eq:late-tail}\\
 \EE(\mathcal T_{n,t}\mid\mathcal F_t)
 &\le Cn^{1/3}/T.\label{eq:late-mean}
\end{align}
\end{lemma}
\begin{proof}
At time $t$, each processed vertex tests at most
$\binom{n-t-1}{d-1}$ still-possible hyperedges.
Ignoring collisions and counting $d-1$ new vertices for each
present edge gives the domination
\[
 X_{n,t}\preceq(d-1)\operatorname{Bin}
                       \left(\binom{n-t-1}{d-1},p\right).
\]
Collisions and already active vertices can only decrease the
true offspring. Uniformly for $t\ge Tn^{2/3}$,
\[
 \EE X_{n,t}
 \le(1+\lambda n^{-1/3}+O(n^{-1}))(1-t/n)^{d-1}
 \le1-cTn^{-1/3}
\]
as soon as $T\ge T_0$.
The exponential transform of $X_{n,t}$ is uniformly finite in a
fixed neighbourhood of the origin, since the binomial mean number
of edges is bounded.

Put $\varepsilon=cTn^{-1/3}$.
The dominating law $(d-1)\operatorname{Bin}(\binom{n-t-1}{d-1},p)$
is integer-valued, has offspring independent from generation to
generation, and has Laplace transform $\EE[e^{\theta X_{n,t}}]$
entire in $\theta$ (because a binomial has a finite exponential
moment whenever $p=O(1/n)$).
These three conditions are exactly those under which Dwass's
cyclic formula applies.
Exponential tilting with a parameter of order $\varepsilon$
gives, uniformly,
\[
 \PP(\mathcal T_{n,t}=j)\le Cj^{-3/2}e^{-c\varepsilon^2j}.
\]
Summing over $j\ge k$ establishes~\eqref{eq:late-tail}.
Finally, for a subcritical Galton--Watson process,
$\EE\mathcal T_{n,t}=(1-\EE X_{n,t})^{-1}$,
which gives~\eqref{eq:late-mean}.
\end{proof}

\begin{lemma}[Active queue at the cutoff time]\label{lem:cutoff}
For every compact set of $\lambda$-values, there exist
$C,c,T_0>0$ such that, for $T\ge T_0$ and $x\ge0$,
\[
 \limsup_{n\to\infty}
 \PP\left(Q_{\lfloor Tn^{2/3}\rfloor}>x\frac{n^{1/3}}T\right)
 \le Ce^{-cx}.
\]
Thus $Q_{\lfloor Tn^{2/3}\rfloor}=O_{\PP}(n^{1/3}/T)$ uniformly
in the double limit $n\to\infty$, then $T\to\infty$.
\end{lemma}
\begin{proof}
First fix $T<\infty$. Theorem~\ref{thm:invariance}, evaluated at
time $T$, gives
\[
 n^{-1/3}Q_{\lfloor Tn^{2/3}\rfloor}
                         \Longrightarrow\R(W_{d,\lambda})(T).
\]
It therefore suffices to obtain an estimate uniform in $T$ for
the limit variable. Put $h=d-1$.
By reversing Brownian increments,
\[
 \R(W_{d,\lambda})(T)\overset{\mathrm{law}}=
 \sup_{0\le r\le T}
 \left\{\sqrt h\,\widetilde B(r)+(\lambda-hT)r+\frac h2r^2\right\},
\]
where $\widetilde B$ is standard Brownian motion.
For $T\ge T_0:=4(1+|\lambda|)/h$ and $0\le r\le T$,
\[
 (\lambda-hT)r+\frac h2r^2\le-\frac{hT}{4}r.
\]
Consequently,
\[
 \R(W_{d,\lambda})(T)\preceq
 \sup_{r\ge0}\left\{\sqrt h\,\widetilde B(r)-\frac{hT}{4}r\right\}.
\]
The supremum of Brownian motion of variance $h$ and drift $-hT/4$
has an exponential distribution with parameter $T/2$:
\[
 \PP\left(\sup_{r\ge0}
       \{\sqrt h\,\widetilde B(r)-hTr/4\}>y\right)=e^{-Ty/2},
 \qquad y\ge0.
\]
Taking $y=x/T$, then applying the Portmanteau theorem to the
closed set $[x/T,\infty)$, gives
\[
 \limsup_{n\to\infty}
 \PP\left(Q_{\lfloor Tn^{2/3}\rfloor}>x\frac{n^{1/3}}T\right)
 \le e^{-x/2}.
\]
The proof is uniform for $\lambda$ in a compact set after a
uniform choice of $T_0$.
\end{proof}

\begin{theorem}[Global excursion convergence in the binomial model]\label{thm:global}
In $H_d(n,p)$,
\[
 n^{-2/3}(|\mathcal C_1|,|\mathcal C_2|,\ldots)
 \Longrightarrow(\xi_1,\xi_2,\ldots)\quad\text{in }\ell^2_\downarrow.
\]
\end{theorem}
\begin{proof}
On $[0,T]$, Theorem~\ref{thm:invariance} and continuity of
reflection give uniform convergence in law to $\R(W_{d,\lambda})$.
Almost surely, this process has no interval of zeros, its finite
excursions are isolated above every positive length threshold,
and their lengths are distinct.
The deterministic excursion-continuity lemma, in the form used
in the construction of the multiplicative coalescent~\cite{Ald97}
and in the metric analysis of~\cite{ABG12},
therefore implies convergence of every finite family of excursions
of length at least $\eta>0$ contained in $[0,T]$.

It remains to eliminate successively $\eta\downarrow0$ and
$T\to\infty$. On $[0,T]$, the sum of squared lengths of excursions
of length at most $\eta$ is deterministically bounded by
\[
 \sum_{i:|I_i|\le\eta}|I_i|^2
       \le\eta\sum_i|I_i|\le\eta T.
\]
The same inequality holds for the discrete excursions after
division by $n^{2/3}$, up to a mesh error $o(1)$.
It therefore allows $\eta$ to tend to zero uniformly after
convergence on the compact interval.
Consider components whose exploration begins after $Tn^{2/3}$.
The size-biasing identity
\[
 \sum_i|\mathcal C_i|^2=\sum_{v=1}^n|\mathcal C(v)|
\]
expresses the sum of squared sizes as a sum over vertices.
For a late vertex $v$, $|\mathcal C(v)|$ is bounded by the total
progeny $\mathcal T_{n,t_v}$ of the Galton--Watson process of
Lemma~\ref{lem:late}, started at the exploration time
$t_v\ge Tn^{2/3}$ of $v$.
Galton--Watson domination and bound~\eqref{eq:late-mean} thus give,
by the law of total expectation,
\[
 \begin{split}
 \EE\left[n^{-4/3}\sum_{\mathcal C_i\ {\rm late}}|\mathcal C_i|^2\right]
 &\le n^{-4/3}\sum_{v\ {\rm late}}\EE|\mathcal C(v)|\\
 &\le n^{-4/3}\cdot n\cdot\frac{Cn^{1/3}}T=\frac CT.
 \end{split}
\]
The component possibly in progress at time $Tn^{2/3}$ is handled
by continuing to the next return of $Q_t$ to zero.
Conditionally on $Q_{\lfloor Tn^{2/3}\rfloor}=q$, the domination
of Lemma~\ref{lem:late}, started with $q$ ancestors, bounds the
number $R_{n,T}$ of vertices still to be explored by a sum of $q$
total progenies. Thus
\[
 \EE(R_{n,T}\mid\mathcal F_{\lfloor Tn^{2/3}\rfloor})
       \le\frac{Cn^{1/3}}T Q_{\lfloor Tn^{2/3}\rfloor}.
\]
Lemma~\ref{lem:cutoff} gives, with uniform quantifiers,
$Q_{\lfloor Tn^{2/3}\rfloor}=O_{\PP}(n^{1/3}/T)$.
On the event $\{Q_{\lfloor Tn^{2/3}\rfloor}\le An^{1/3}/T\}$,
conditional Markov's inequality gives, for every $\delta>0$,
\[
 \PP(R_{n,T}>\delta n^{2/3}\mid\mathcal F_{\lfloor Tn^{2/3}\rfloor})
                   \le\frac{CA}{\delta T^2}.
\]
Combining this estimate with Lemma~\ref{lem:cutoff} gives explicitly
\[
 \limsup_{n\to\infty}\PP(R_{n,T}>\delta n^{2/3})
                \le Ce^{-cA}+\frac{CA}{\delta T^2}.
\]
First choose $A$ sufficiently large, then $T$ sufficiently large.
It follows that
\[
 \lim_{T\to\infty}\limsup_{n\to\infty}
              \PP(R_{n,T}>\delta n^{2/3})=0.
\]
The part of this component before the cutoff time already belongs
to the excursions observed on $[0,T]$; only the residual length
$R_{n,T}$ must therefore be added.
Its effect on the $\ell^2_\downarrow$ norm tends to zero with $T$.
By Markov's inequality,
\[
 \lim_{T\to\infty}\limsup_{n\to\infty}
 \PP\left\{n^{-4/3}\sum_{\mathcal C_i\ {\rm late}}|\mathcal C_i|^2
                                                        >\delta\right\}=0.
\]
The same property holds for the limiting process, whose parabolic
drift tends to $-\infty$.
Excursion convergence on compact intervals, together with these
two tail estimates, gives convergence in $\ell^2_\downarrow$.
\end{proof}

\begin{theorem}[Transfer to the fixed-edge model]\label{thm:fixed}
For $M$ given by~\eqref{eq:window},
\[
 n^{-2/3}L_n\Longrightarrow\xi_1
\]
in $H_d(n,M)$.
\end{theorem}
\begin{proof}
For $\theta\in\mathbb R$, put
\[
 m_n(\theta)=\left\lfloor\frac{n}{d(d-1)}(1+\theta n^{-1/3})\right\rfloor,
 \qquad p_n(\theta)=\frac{m_n(\theta)}N.
\]
Fix $\varepsilon>0$.
Assign independent uniform $[0,1]$ variables $U_E$ to the $N$
possible hyperedges and order the edges by increasing values of
$U_E$. For $p\in[0,1]$, put
\[
 H(p)=\{E:U_E\le p\},\qquad E(p)=|H(p)|.
\]
Then $H(p)$ has exactly the law $H_d(n,p)$; the edge order is a
uniform permutation; and conditionally on $E(p)=r$, $H(p)$ is
uniform among hypergraphs with $r$ edges.
On the same space, the hypergraph formed by the first $M$ edges
has exactly the law $H_d(n,M)$.

Put $p_-=p_n(\lambda-\varepsilon)$, $p_+=p_n(\lambda+\varepsilon)$
and $E_-=E(p_-)$, $E_+=E(p_+)$.
Simultaneously, $H(p_-)\subseteq H(p_+)$ and $E_-\le E_+$.
It follows that
\[
 \EE E_\pm=m_n(\lambda\pm\varepsilon),\qquad
 \Var(E_\pm)=Np_\pm(1-p_\pm)=O(n).
\]
Since $M=m_n(\lambda)+O(1)$ and
\[
 m_n(\lambda\pm\varepsilon)-M
       =\pm\frac{\varepsilon}{d(d-1)}n^{2/3}+O(1),
\]
Chebyshev's inequality, or a Chernoff bound, gives
\[
 \PP\{E_-\le M\le E_+\}\longrightarrow1,
\]
because the gap between their means and $M$ has order $n^{2/3}$,
whereas their standard deviation has order $n^{1/2}$.
On this event, the coupling by the uniform order satisfies
\[
 H(p_-)\subseteq H_d(n,M)\subseteq H(p_+),
\]
and consequently
\[
 L_n(\lambda-\varepsilon)\le L_n(M)\le L_n(\lambda+\varepsilon).
\]

Theorem~\ref{thm:global} identifies the two extreme laws.
Under a common Brownian motion, $W_{d,\lambda\pm\varepsilon}$
converges uniformly on every compact interval to $W_{d,\lambda}$
as $\varepsilon\downarrow0$.
The late-excursion control used in the preceding proof makes this
continuity global; the longest excursion is almost surely unique.
Thus
\[
 \xi_1(\lambda\pm\varepsilon)\Longrightarrow\xi_1(\lambda).
\]
For every continuous bounded increasing function $f$, the coupling gives
\[
 \begin{split}
 \EE f(n^{-2/3}L_n(\lambda-\varepsilon))-o(1)
 &\le\EE f(n^{-2/3}L_n(M))\\
 &\le\EE f(n^{-2/3}L_n(\lambda+\varepsilon))+o(1).
 \end{split}
\]
First let $n$ tend to infinity, then $\varepsilon\downarrow0$.
Continuous bounded increasing functions determine laws on
$\mathbb R$, completing the proof.
\end{proof}

\begin{remark}[Absence of local conditioning]
The preceding proof never divides by $\PP\{E(p)=M\}$ and therefore
requires no binomial local theorem.
It uses only monotonicity of the largest component size under edge
addition, concentration of the two edge counts, and continuity
in $\lambda$ of the limiting law.
The argument applies unchanged to every fixed integer $d\ge2$.
\end{remark}
Lemma~\ref{lem:finite} and Theorem~\ref{thm:invariance} close
finite-horizon convergence without hypotheses.
This result does not suffice for the infinite ordered sequence of
excursions: tightness beyond every horizon and control of
component starts are still needed.

\section{Wright normalization and identification with Pittel}
\subsection{Correct normalization of Wright's constants}
Two families bear the name Wright constants. The coefficients $c_\ell$
of the formal kernel series have factorial growth~\cite{Jan07}.
The normalized coefficients $\gamma_\ell$ (written $\gamma_{\ell,d}$
when one wishes to recall the dependence on $d$; for $d=2$ the two
notations coincide), which enter the critical law, are defined by
\begin{equation}\label{eq:gamma-def}
 C_2(k,k+\ell)=(1+o(1))\gamma_\ell k^{k-1/2+3\ell/2}\qquad(k\to\infty),
\end{equation}
and satisfy~\cite{Wri77,Wri78,Wri80,VKM09}
\begin{equation}\label{eq:gamma-growth}
 \gamma_\ell=(1+o(1))\sqrt{\frac3{4\pi}}
                         \left(\frac e{12\ell}\right)^{\ell/2}.
\end{equation}
Consequently,
\begin{equation}\label{eq:G}
 G(s)=\sum_{\ell\ge-1}\gamma_\ell s^\ell
\end{equation}
converges absolutely for every $s\ne0$.
This series is related to the moment generating function of Brownian
excursion area~\cite{Spe97,VKM09}.
The factor $\Gamma(3\ell/2)$ appears in the coefficient asymptotics
(Theorem~\ref{thm:wright}), where it occurs as
$b_{\ell,d}/\Gamma(3\ell/2)$, and this term is precisely
$\gamma_\ell\cdot(1+o(1))$.
It is this cancellation between the factorial growth of $b_{\ell,d}$
and the factor $\Gamma(3\ell/2)$ that produces a $\gamma_\ell$ with
subexponential growth~\eqref{eq:gamma-growth}, making series~\eqref{eq:G}
entire. The summation criterion~\eqref{eq:excess-bound}, meanwhile,
uses $\Gamma(\ell/2+1)$, a weaker factor that would suffice for the
Weierstrass test. This uniform estimate is neither asserted nor used
in the proof of the main theorem.

\begin{proposition}[Compatibility of the critical normalizations]\label{prop:normalization}
Put $h=d-1$ and
\[
 m_n(\theta)=\left\lfloor\frac n{d(d-1)}(1+\theta n^{-1/3})\right\rfloor,
 \qquad p_n(\theta)=\frac{m_n(\theta)}{\binom nd}.
\]
Then, uniformly for $\theta$ in a compact set,
\begin{equation}\label{eq:pn}
 p_n(\theta)=\frac{1+\theta n^{-1/3}+O(n^{-1})}{h\binom{n-1}{d-1}}.
\end{equation}
Consequently, exploration at scale $t=\lfloor sn^{2/3}\rfloor$ has limit
\[
 W_{d,\theta}(s)=\sqrt h\,B(s)+\theta s-\frac h2s^2.
\]
After the time change $s=h^{-1/3}u$,
\begin{equation}\label{eq:brownian-scale}
 W_{d,\theta}(h^{-1/3}u)\overset{\mathrm{law}}=
 h^{1/3}\left\{B(u)+\theta h^{-2/3}u-\frac12u^2\right\}.
\end{equation}
Thus the three normalizations used in the manuscript (critical window,
$n^{2/3}$ scale and Pittel parameters) are compatible without an
additional implicit constant.
\end{proposition}
\begin{proof}
The identity
\[
 \frac n{d(d-1)}\binom nd^{-1}=\frac1{(d-1)\binom{n-1}{d-1}}
\]
and the rounding error $O(1)$ in $m_n(\theta)$ give~\eqref{eq:pn};
indeed, this error becomes relative $O(n^{-1})$ since $m_n(\theta)$
has order $n$. The conditional expansions~\eqref{eq:drift}--\eqref{eq:variance}
and Theorem~\ref{thm:invariance} then give the displayed process.
Finally, self-similarity
$B(h^{-1/3}u)\overset{\mathrm{law}}=h^{-1/6}B(u)$ implies
\[
 \sqrt h\,B(h^{-1/3}u)=h^{1/3}B(u),\qquad
 \theta h^{-1/3}u=h^{1/3}(\theta h^{-2/3}u),
\]
and $(h/2)h^{-2/3}u^2=h^{1/3}u^2/2$, proving~\eqref{eq:brownian-scale}.
\end{proof}

\subsection{Exact reduction of the hypergraph law to Pittel's law}
Let $\FP(x,\mu)$ denote the distribution function of the longest
excursion interval of
\[
 B(t)+\mu t-\frac12t^2.
\]
Pittel~\cite[equation (3.25)]{Pit01} gives an explicit integral
representation and proves that it is the limiting law of the largest
component of the critical graph. This notation fixes the published
integral without ambiguity, including its amplitude and contour, and
avoids reconstructing these constants from a merely formal phase.

\begin{theorem}[Identification with Pittel's law by Brownian self-similarity]\label{thm:pittel}
For every $d\ge2$, $a>0$ and $\lambda\in\mathbb R$, putting $h=d-1$,
the limit variable of Theorem~\ref{thm:excursions} satisfies
\begin{equation}\label{eq:pittel}
 \PP\{\xi_1\le a\}=\FP(h^{1/3}a,\lambda h^{-2/3}).
\end{equation}
Thus, after the substitution
\[
 a_{\mathrm{Pit}}=h^{1/3}a,\qquad \lambda_{\mathrm{Pit}}=\lambda h^{-2/3},
\]
Pittel's integral representation identifies the limit law for every
$d\ge2$, without any additional hypothesis on hypergraph kernels.
\end{theorem}
\begin{proof}
Proposition~\ref{prop:normalization}, with $\theta=\lambda$, gives
exactly rescaling~\eqref{eq:brownian-scale}.
Vertical multiplication by the positive constant $h^{1/3}$ does not
change the zeros or excursion intervals of the reflection.
The time change $s=h^{-1/3}t$ is a deterministic increasing Lipschitz
map; the Skorokhod reflection operator commutes with continuous
deterministic time changes (see~\cite[Chapter VI]{EthierKurtz86}),
so excursion lengths are multiplied by $h^{-1/3}$ without altering
their order or distribution. Consequently,
\[
 \xi_1(d,\lambda)\overset{\mathrm{law}}=
 h^{-1/3}\xi_1(2,\lambda h^{-2/3}).
\]
The case $d=2$ is precisely the variable identified by Pittel's
integral formula~\cite{Pit01}; this gives~\eqref{eq:pittel}.
\end{proof}

\begin{remark}[Attribution of the integral equality]
Identification of the limit law with Pittel's integral formula rests
on two distinct contributions.
The equality between the law of the largest excursion of
$B(t)+\mu t-\frac12t^2$ and the limit of the largest component of the
critical random graph is due to Pittel~\cite{Pit01} and Aldous~\cite{Ald97}.
The specific contribution of the present article lies in the rescaling
factors $h^{1/3}$ and $\lambda h^{-2/3}$ that extend the formula to
$d$-uniform hypergraphs for every $d\ge2$, with $h=d-1$.
\end{remark}

\begin{corollary}[Main theorem by the probabilistic route]\label{cor:main}
For every fixed integer $d\ge2$, every $\lambda\in\mathbb R$ and
every integer sequence
\[
 M=\frac n{d(d-1)}(1+\lambda n^{-1/3})+O(1),
\]
if $L_n$ is the size of the largest component of $H_d(n,M)$, then,
at every continuity point $a>0$ of the distribution function,
\begin{equation}\label{eq:main-closed}
 \lim_{n\to\infty}\PP\{n^{-2/3}L_n\le a\}
 =\FP\bigl((d-1)^{1/3}a,\lambda(d-1)^{-2/3}\bigr).
\end{equation}
\end{corollary}
\begin{proof}
Theorem~\ref{thm:invariance} gives convergence of the exploration
process on every compact interval.
Lemma~\ref{lem:late} and Theorem~\ref{thm:global} ensure tightness
of the ordered excursions and passage to $\ell^2_\downarrow$.
Theorem~\ref{thm:fixed} transfers the longest excursion from the
binomial model to the $M$-edge model.
Finally, Theorem~\ref{thm:pittel} applies Brownian self-similarity and
gives exactly the right-hand side of~\eqref{eq:main-closed}.
The Portmanteau theorem turns convergence in distribution into
convergence of distribution functions at continuity points.
\end{proof}

\begin{remark}[Precise meaning of ``without an analytic hypothesis'']
This expression means only that the proof of Corollary~\ref{cor:main}
does not use the hypergraph contour derivation or the series $g_d$.
It does not, however, dispense with verification of the probabilistic
estimates in Lemmas~\ref{lem:finite} and~\ref{lem:late}, excursion
continuity, and transfer to the fixed-edge model, which are the
explicit ingredients of the proof.
\end{remark}

\begin{corollary}[A differential equation for verification]
Setting
\[
 \Phi(x;\mu)=\frac1{x\sqrt{2\pi}}G(x^{3/2})
       \exp\left\{-\frac{\mu^3}{6}+\frac{(\mu-x)^3}{6}\right\},
\]
one obtains
\begin{equation}\label{eq:ode}
 \partial_x\FP(x,\mu)=\FP(x,\mu-x)\Phi(x;\mu).
\end{equation}
This identity, proved by the local theorem of~\cite{VKM09},
provides an independent check on the normalization and positivity
of the density.
\end{corollary}

\section{The $d$-uniform Bollob\'as bound and the analytic domain}
The Bollob\'as bound is formulated here with the proved quantifiers,
in two regimes: Wright for each fixed excess, combinatorial throughout
the domain. It does not enter the proof of Theorem~\ref{thm:main}.

\begin{proposition}[Verified Puiseux expansion]
Let $h=d-1$ and let $T_d(x)$ be defined by $T_d=x\exp(T_d^h/h!)$.
If $\tau_d$ satisfies $h\tau_d^h/h!=1$ and
$\rho_d=\tau_d\exp(-\tau_d^h/h!)$, then
\[
 T_d(x)=\tau_d-\tau_d\sqrt{\frac2h}\sqrt{1-x/\rho_d}+O(1-x/\rho_d).
\]
\end{proposition}
\begin{proof}
Put $x(T)=T\exp(-T^h/h!)$. We have $x'(\tau_d)=0$ and
\[
 \frac{x''(\tau_d)}{\rho_d}=-\frac h{\tau_d^2}.
\]
Taylor's formula gives
$1-x/\rho_d=\frac{h}{2\tau_d^2}(T-\tau_d)^2+O((T-\tau_d)^3)$.
Inversion on the real increasing branch, followed by continuation
into a dented domain, gives the stated expansion.
It agrees with the singular inversion scheme of~\cite[Theorem VI.6]{FS09}.
\end{proof}

\begin{theorem}[$d$-uniform Bollob\'as bound]\label{thm:bollobas}
Fix $d\ge2$ and put $h=d-1$.

\emph{(i) Wright regime.}
For each fixed integer $\ell\ge1$, there exist $K_{d,\ell}<\infty$
and $v_0(d,\ell)$ such that, for every $v\ge v_0(d,\ell)$ satisfying
$v+\ell\equiv0\pmod h$,
\begin{equation}\label{eq:wright-bound}
 C_d\left(v,\frac{v+\ell}{h}\right)
       \le K_{d,\ell}v!\rho_d^{-v}v^{3\ell/2-1}.
\end{equation}

\emph{(ii) Global domain.}
For all admissible $v\ge1$ and $\ell\ge1$, putting $q=(v+\ell)/h$,
\begin{equation}\label{eq:global}
 C_d(v,q)\le\binom{\binom vd}{q}
         \le\left(\frac{e\binom vd}{q}\right)^q.
\end{equation}
\end{theorem}
\begin{proof}
For (i), Theorem~\ref{thm:wright} gives, for fixed $\ell$ and along
admissible classes,
\[
 C_d\left(v,\frac{v+\ell}{h}\right)
 =\gamma_{\ell,d}v!\rho_d^{-v}v^{3\ell/2-1}
                  \bigl(1+O_{d,\ell}(v^{-1/2})\bigr).
\]
It suffices to choose $v_0(d,\ell)$ so that the absolute value of the
relative remainder is at most $1$, then $K_{d,\ell}=2\gamma_{\ell,d}$.
For (ii), every connected hypergraph counted by $C_d(v,q)$ is a choice
of $q$ edges from the $\binom vd$ possible edges.
The second inequality is $\binom Nq\le(eN/q)^q$.
\end{proof}

\begin{remark}[Attribution and scope]
The name ``Bollob\'as bound'' refers to the strategy of
uniform majorization from connected graph enumeration.
Theorem~\ref{thm:bollobas} gives a rigorously valid $d$-uniform
extension in two-regime form.
It does not assert the existence of constants $A_d,B_d$ independent
of $\ell$ in a Wright bound valid up to the dense regime.
Such a formula would be false for $d\ge4$ when $\ell$ has order $v^d$.

The two parts of Theorem~\ref{thm:bollobas} have distinct domains and
functions; they are not combined to announce a uniform Wright-type
bound. Consequently, no unjustified summation over $\ell$ and no
interchange of the operators $\lim$ and $\sum$ occurs in the proof
of the main result.
\end{remark}

\section{Assessment of the two routes}
Corollary~\ref{cor:main} closes the limit law unconditionally by the
probabilistic route: convergence of the exploration process on every
finite horizon (Theorem~\ref{thm:invariance}), tightness of ordered
excursions (Theorem~\ref{thm:global}), transfer to the $M$-edge model
(Theorem~\ref{thm:fixed}), and identification of the law by Brownian
self-similarity (Theorem~\ref{thm:pittel}).
No uniformity in the excess enters.

The analytic route, for its part, produces exact identities and
normalization cancellations that the probabilistic route does not
establish. The formulas in Theorem~\ref{thm:exact},
Proposition~\ref{prop:trees}, and Theorem~\ref{thm:unicyclic} hold
for every fixed $d\ge2$; their constants and normalizations are
verifiable and will serve subsequent work.
What Section 9 delineates is precisely what the analytic route
proves today, and what it does not yet prove without an additional
hypothesis on hypergraph kernels.

\section*{Acknowledgements}

The author is grateful to Professor Vlady Ravelomanana (IRIF, CNRS,
Universit\'e Paris Cit\'e) for his careful reading of an earlier
draft, his precise remarks on the analytic combinatorics arguments,
and his encouragement to pursue the $d$-uniform generalisation beyond
the classical $d=2$ case.

\end{document}